\documentclass[12pt,a4paper]{article}

\usepackage{amssymb}
\usepackage{amsmath, amsthm}
\usepackage{arydshln}
\usepackage{epsfig}
\usepackage{setspace}

\usepackage{geometry}

\def\RR{{\bf R}}
\def\ZZ{{\bf Z}}

\def\bigmid {\ \left|{\Large \strut}\right.}

\numberwithin{equation}{section}

\newcommand{\Argmin}{\mathop{\rm Argmin} }

\newtheorem{Thm}{Theorem}[section]
\newtheorem{Prop}[Thm]{Proposition}
\newtheorem{Lem}[Thm]{Lemma}

\theoremstyle{definition}
\newtheorem*{Clm}{Claim}

\newtheorem{Ex}[Thm]{Example}

\title{Uniform semimodular lattices 
	and valuated matroids}
\author{Hiroshi HIRAI \\
Department of Mathematical Informatics, \\
Graduate School of Information Science and Technology,   \\
The University of Tokyo, Tokyo, 113-8656, Japan.\\
\texttt{\normalsize hirai@mist.i.u-tokyo.ac.jp}
}
\begin{document}

\maketitle
\begin{abstract}
	In this paper, we present a lattice-theoretic characterization for valuated matroids, which is an extension of
	the well-known cryptomorphic equivalence between matroids 
	and geometric lattices ($=$ atomistic semimodular lattices).
	We introduce a class of semimodular lattices, called uniform semimodular lattices, and establish a cryptomorphic 
	equivalence between integer-valued valuated matroids 
	and uniform semimodular lattices. 
	Our result includes a coordinate-free lattice-theoretic characterization
	of integer points in tropical linear spaces,  incorporates  
	the Dress-Terhalle completion process of valuated matroids, 
	and  establishes a smooth connection with Euclidean buildings of type A.
\end{abstract}


\section{Introduction}
Matroids can be characterized 
by various cryptomorphically equivalent axioms; see e.g.,~\cite{Aigner}. 
Among them, a lattice-theoretic characterization by Birkhoff~\cite{Birkhoff} 
is well-known: 
The lattice of flats of any matroid
is a {\em geometric lattice} ($=$ atomistic semimodular lattice), and
any geometric lattice gives rise to a simple matroid.

The goal of the present article is to extend 
this classical equivalence to {\em valuated matroids} 
(Dress and Wenzel~\cite{DressWenzel_greedy,DressWenzel}).
Valuated matroid is a quantitative generalization 
of matroid, which abstracts linear dependence structures of vector spaces 
over a field with a non-Archimedean valuation. 
A valuated matroid is defined as 
a function on the base family of a matroid
satisfying a certain exchange axiom that originates from the Grassmann-Pl\"ucker identity.
Just as matroids, valuated matroids enjoy nice optimization properties; 
they can be optimized by a greedy algorithm, 
and this property characterizes valuated matroids.
In the literature of combinatorial optimization,
the theory of valuated matroids has evolved into  
{\em discrete convex analysis}~\cite{MurotaBook}, 
which is a framework of ``convex" functions on discrete structures 
generalizing matroids and submodular functions.
In tropical geometry (see e.g., \cite{TropicalBook}), 
a valuated matroid 
is called a {\em tropical Pl\"ucker vector}.
The space of valuated matroids is understood as
a tropical analogue of the Grassmann variety in algebraic geometry;
see \cite{Speyer08,SpeyerSturmfels04}.

While Murota and Tamura~\cite{MurotaTamura01} established 
cryptomorphically equivalent axioms for valuated matroids
in terms of (analogous notions of) {\em circuits}, {\em cocircuits}, {\em vectors}, and {\em covectors}, a lattice-theorectic axiom has never been given in the literature.
The goal of this paper is to establish a lattice-theoretic axiom for valuated matroids 
by introducing a new class of semimodular lattices,
called {\em uniform semimodular lattices}.
This class of lattices can be 
 viewed as an affine-counterpart of geometric lattices, and is defined by a fairly simple axiom:
They are semimodular lattices 
with the property that the operator 
$x \mapsto $ (the join of all elements covering $x$)
is an automorphism.
This operator was introduced in a companion paper~\cite{HH18a} 
to characterize 
Euclidean buildings in a lattice-theoretic way.

The main result of this paper is 
a cryptomorphic equivalence 
between uniform semimodular lattices and integer-valued valuated matroids. 
The contents of this equivalence and its intriguing features 
are summarized as follows:
\begin{itemize}
	\item
	A valuated matroid is constructed from 
	a uniform semimodular lattice ${\cal L}$ as follows. 
	We introduce the notion of a {\em ray} and {\em end} in ${\cal L}$.
	A ray is a chain of ${\cal L}$ with a certain straightness property, and
    an end is an equivalence class of the parallel relation on the set of rays.
    Ends will play the role of atoms in a geometric lattice. 
	We introduce a matroid ${\bf M}^{\infty}$ on the set $E$ of ends, 
	called the {\em matroid at infinity}, 
	which will be the underlying matroid of our valuated matroid. 
	As expected from the name, 
	this construction is inspired by 
	the {\em spherical building at infinity} in a Euclidean building.
	A $\ZZ^n$-sublattice ${\cal S}(B)$ ($\simeq \ZZ^n$) is naturally associated with each base $B$ in ${\bf M}^{\infty}$, and plays the role of 
	apartments in a Euclidean building.
	Then a valuated matroid $\omega = \omega^{{\cal L},x}$ on $E$ is 
	defined from apartments and any fixed $x \in {\cal L}$; 
	the value $\omega(B)$ 
	is the negative of a ``distance" between $x$ and ${\cal S}(B)$.
	It should be emphasized that this construction is done purely in a coordinate-free  lattice-theoretic manner.
	
	\item  The reverse construction of a 
	uniform semimodular lattice from a valuated matroid uses 
	the concept of the {\em tropical linear space}. 
	The tropical linear space ${\cal T}(\omega)$ is a polyhedral object in $\RR^E$ associated with a valuated matroid $\omega$ on $E$. 
	This concept and the name were introduced by Speyer~\cite{Speyer08} 
	in the literature of tropical geometry. 
	Essentially equivalent concepts were earlier considered 
	by Dress and Terhalle~\cite{DressTerhalle93,DressTerhalle98}  
	as the {\em tight span}
	and by Murota and Tamura~\cite{MurotaTamura01} 
	as the {\em space of covectors}.
	In the case of a matroid (i.e., $\{0,-\infty\}$-valued valuated matroid), the tropical linear space reduces to the {\em Bergman fan} of the matroid,   
	which is viewed as a geometric realization of the order complex of 
	the geometric lattice of flats~\cite{ArdilaKlivans06}. 
	We show that the set ${\cal L}(\omega) := {\cal T}(\omega) \cap \ZZ^E$ of integer points 
	in ${\cal T}(\omega)$ forms a uniform semimodular lattice. 
	Then the original $\omega$ 
	is recovered by the above construction (up to the projective-equivalence), 
	and ${\cal T}(\omega)$ is a geometric realization 
	of a special subcomplex of the order complex of ${\cal L}(\omega)$.
	Thus our result establishes
	a  coordinate-free lattice-theoretic characterization 
	of tropical linear spaces.
	
	\item The above constructions incorporate, in a natural way,  
	the {\em completion process of valuated matroids} by Dress and Terhalle~\cite{DressTerhalle93}, 
	which is 
	a combinatorial generalization of the $p$-adic completion.
	They introduced an ultrametric metrization of the underlying set $E$
	by a valuated matroid $\omega$, and 
	a completeness concept for valuated matroids 
	in terms of the completeness of this metrization of $E$.
	They proved that any (simple) valuated matroid $(E, \omega)$ is (uniquely) extended to 
	a complete valuated matroid $(\bar E, \bar \omega)$, which is called  
	a {\em completion} of $(E, \omega)$.

	We show that the space $E$ of ends in a uniform semimodular lattice ${\cal L}$
	admits an ultrametric metrization $d$, and it is complete in this metric, 
	where $d$ coincides with the Dress-Terhalle metrization of
	the constructed valuated matroid $\omega^{{\cal L},x}$.
	Then the process $\omega \to {\cal L}(\omega) \to \omega^{{\cal L}(\omega),x}$ coincides with 
	the Dress-Terhalle completion of~$\omega$. 
	
	\item Our result sums up, from a lattice-theory side, 
	connections between 
	valuated matroids and {\em Euclidean buildings} 
	(Bruhat and Tits~\cite{BruhatTits}), pointed out by~\cite{DressKahrstromMoulton11,DressTerhalle98,JoswigSturmfelsYu07}; see also a recent work~\cite{Zhang2018}.   
    Let us recall the spherical situation, and 
    recall a {\em modular matroid}, which is a matroid 
    whose lattice of flats is a modular lattice.
    We can say that 
    a modular matroid is equivalent to a {\em spherical building of type A}~\cite{Tits}.
    Indeed, a classical result of Birkhoff~\cite{Birkhoff} says that
    a modular geometric lattice is precisely the direct product of subspace lattices of 
    projective geometries. Another classical result by Tits~\cite{Tits} says 
    that a spherical building of type A is the order complex of 
    the direct product of subspace lattices of projective geometries. 
    
    An analogous relation is naturally established for valuated matroids
    by introducing the notion of  
    a {\em modular valuated matroid}, which is defined as a valuated matroid 
    such that the corresponding uniform semimodular lattice is a modular lattice.
    The companion paper~\cite{HH18a} showed that 
    uniform modular lattices are cryptomorphically equivalent to Euclidean buildings of type A.
    Thus a modular valuated matroid $\omega$ 
    is equivalent to a Euclidean building of type A, 
    in which (the projection of) the tropical linear space ${\cal T}(\omega)$ 
    is a geometric realization of the Euclidean building.
	This generalizes a result by Dress and Terhalle~\cite{DressTerhalle98}
	obtained for the Euclidean building of ${\rm SL}(F^n)$ for a valued field $F$.
\end{itemize}

The rest of this paper is organized as follows.
Sections~\ref{sec:pre} and \ref{sec:valuated} are preliminary sections 
on lattice, (valuated) matroids, and tropical linear spaces.
Section~\ref{sec:uniform} 
constitutes the main body of our results on uniform semimodular lattices. 
Section~\ref{sec:example} discusses 
representative examples of valuated matroids 
in terms of uniform semimodular lattices.

\section{Preliminaries}\label{sec:pre}

Let $\RR$ denote the set of real numbers. 
Let $\ZZ$ and $\ZZ_+$ denote the sets of integers and nonnegative integers, respectively.
For a set $E$ (not necessarily finite), 
let $\RR^E$, $\ZZ^E$, and $\ZZ_+^E$
denote the sets of all functions from $E$ to $\RR$, $\ZZ$, and $\ZZ_+$, respectively.
A function $g: E \to \ZZ$ is said to be {\em upper-bounded} 
if there is $M \in \ZZ$ such that $g(e) \leq M$ for all $e \in E$.
If $|g(e)| \leq M$  for all $e \in E$,  then $g$ is said to be {\em bounded}.
Let ${\bf 1}$ denote the all-one vector in $\RR^E$, i.e., ${\bf 1}(e) = 1$ $(e \in E)$. 
For a subset $F \subseteq E$, let ${\bf 1}_F$ denote the incidence vector 
of $F$ in $\RR^E$, i.e., ${\bf 1}_{F}(e) = 1$ if $e \in F$ and zero otherwise.
${\bf 1}_{\{e\}}$ is simply denoted by ${\bf 1}_{e}$.
Let ${\bf 0}$ denote the zero vector.
For $x,y \in \RR^E$, let $\min (x,y)$ and $\max (x,y)$ 
denote the vectors in $\RR^E$ obtained from $x,y$
by taking componentwise minimum and maximum, respectively; namely  
$\min(x,y)(e) = \min (x(e),y(e))$ and $\max(x,y)(e) = \max (x(e),y(e))$ for $e \in E$.
The vector order $\leq$ on $\RR^E$ is defined by
$x \leq y$ if $x(e) \leq y(e)$ for all $e \in E$.
For $e \in E$ and $B \subseteq E$, 
we denote $B \cup \{e\}$ and $B \setminus \{e\}$
by $B + e$ and $B - e$, respectively.

\subsection{Lattices}\label{subsec:lattice}
We use the standard terminology on posets (partially ordered sets) 
and lattices (see, e.g., \cite{Aigner,Birkhoff}), where $\preceq$ denotes a partial order relation, and 
$x \prec y$ means $x \preceq y$ and $x \neq y$.
A lattice is a poset 
${\cal L}$ such that every pair $x,y$ has 
the greatest common lower bound $x \wedge y$ and 
the least common upper bound $x \vee y$; 
the former is called the {\em meet} of $x,y$, and 
the latter is called the {\em join} of  $x,y$.
For a subset $S \subseteq {\cal L}$, 
the greatest lower bound of $S$ (the {\em meet} of $S$) is denoted by 
$\bigwedge S$ (if it exists), and the least upper bound of $S$ (the {\em join} of $S$) 
is denoted by $\bigvee S$ (if it exists). 
For elements $x,y$ with $x \preceq y$, the {\em interval} $[x,y]$ of $x,y$ 
is the set of elements $z$ with $x \preceq z \preceq y$.
If  $[x,y] = \{x,y\}$ and $x \neq y$, then we say that $y$ {\em covers} $x$ 
and write $x \prec_1 y$.
A {\em chain} is a totally ordered subset ${\cal C}$ of ${\cal L}$;  
a chain will be written, say, as 
$x_0 \prec x_1 \prec \cdots \prec x_m \prec \cdots$.
The {\em length} of chain ${\cal C}$ is defined as its cardinality $|{\cal C}|$ minus one.
In this paper, we deal with lattices satisfying the following finiteness assumption:
\begin{itemize}
	\item[(F)] No interval $[x,y]$ has a chain of infinite length.
\end{itemize}
An order-preserving bijection $\varphi:{\cal L} \to {\cal L}'$ 
is called an {\em isomorphism}.
If ${\cal L} = {\cal L'}$, then an isomorphism $\varphi$ 
is called an {\em automorphism} on ${\cal L}$.
A {\em sublattice} of a lattice ${\cal L}$ is a subset ${\cal L'} \subseteq {\cal L}$
with the property that $x,y \in {\cal L'}$ imply $x \wedge y, x \vee y \in {\cal L}'$.
Intervals are sublattices.
An {\em atom} is an element that covers 
the minimum ${\bar 0} = \bigwedge {\cal L}$.
The rank of ${\cal L}$ (having the minimum and maximum) is defined as  
the maximum length of a maximal chain of ${\cal L}$.
A {\em height function} of a lattice ${\cal L}$ is 
an integer-valued function $r:{\cal L} \to \ZZ$ 
such that 
$r(y) = r(x) + 1$ for any $x,y \in {\cal L}$ with $x \prec_1 y$.

A lattice ${\cal L}$ is said to be {\em semimodular} 
if 
$x \wedge a \prec_1 a$ implies $x \prec_1 x \vee a$  for any $x,a \in {\cal L}$. 
From the definition, we easily see that
a semimodular lattice satisfies the Jordan-Dedekind chain condition:
\begin{itemize}
	\item[(JD)] For any interval $[x,y]$, all maximal chains in $[x,y]$ have the same length. 
\end{itemize}
We denote this length by $r[x,y]$, which is finite by (F).

\begin{Lem}\label{lem:semimodular}
	For a lattice ${\cal L}$, the following conditions are equivalent:
	\begin{itemize}
		\item[{\rm (1)}] ${\cal L}$ is semimodular.
		\item[{\rm (2)}] For $a,b \in {\cal L}$, if $a,b$ cover $a \wedge b$, then $a \vee b$ covers $a,b$.
		\item[{\rm (3)}] ${\cal L}$ admits a height function $r$ satisfying 
		\begin{equation}\label{eqn:semimo}
		r(x) + r(y) \geq r(x \wedge y) + r(x \vee y) \quad (x,y \in {\cal L}).
		\end{equation}
	\end{itemize}
\end{Lem} 
\begin{proof}[Sketch of proof]
We verify (1) $\Rightarrow$ (3); other directions are easy or obvious.	
Fix $z \in {\cal L}$, 
define $r:{\cal L} \to \ZZ$ by  $r(x) := r[z, x \vee z] - r[x, x \vee z]$.
Consider an element $y$ that covers $x$.
If $y \preceq x \vee z$, then $x \vee z  = y \vee z$ and 
$r[y,y \vee z] = r[x,x \vee z] - 1$.
If $y \not \preceq x \vee z$, then
by semimodularity, $y \vee z$ covers $x \vee z$, 
and hence $r[y,y \vee z] = r[x,x \vee z]$ and  $r[z,y \vee z] = r[z,x \vee z] +1$.
Thus $r$ is a height function.

We show (\ref{eqn:semimo}).
Consider a maximal chain 
$x \wedge y = z_0 \prec_1 z_1 \prec_1 \cdots \prec_1 z_k = y$, 
where $k = r[x\wedge y, y]$ by (JD).
Consider a chain $x = x \vee z_0 \preceq  x \vee z_1 \preceq \cdots \preceq x\vee z_k = x \vee y$, which contains a maximal chain in $[x, y]$ 
by the semimodularity.
This implies $r(x \vee y) - r(x) = r[x, x \vee y] \leq k = r[x \wedge y, y] = r(y) - r(x \wedge y)$, and (\ref{eqn:semimo}).
\end{proof}
A {\em modular pair} is a pair $x,y \in {\cal L}$ satisfying (\ref{eqn:semimo})
in equality.
A {\em geometric lattice} is a semimodular lattice  
such that it has the minimum and maximum, and every element can be represented as 
the join of atoms. 
A {\em hyperplane} in a geometric lattice is 
an element that is covered by the maximum element.
The following is well-known.
\begin{Lem}[{See, e.g., \cite[Section II.3]{Aigner}}]\label{lem:hyperplane}
	Let ${\cal L}$ be a geometric lattice.
	\begin{itemize}
		\item[{\rm (1)}] Every element in ${\cal L}$ is 
		written as the meet of hyperplanes.
		\item[{\rm (2)}] Every interval in ${\cal L}$ is a geometric lattice.
	\end{itemize}
\end{Lem}
A {\em modular lattice} is a lattice ${\cal L}$ such that 
for every triple $x,y,z \in {\cal L}$ with $x \preceq z$ 
it holds $x \vee (y \wedge z) = (x \vee y) \wedge z$.
A modular lattice is precisely a semimodular lattice in which every pair is modular.

\subsection{Matroids}\label{subsec:matroid}
Here we introduce matroids on a possibly infinite ground set, where 
our treatment follows \cite[Chapter VI]{Aigner}.
 A {\em matroid} ${\bf M} = (E,{\cal I})$ is 
 a pair of a set $E$
 and a family ${\cal I}$ of subsets of $E$ 
 such that $\emptyset \in {\cal I}$, 
 $I' \subseteq I \in  {\cal I}$ implies $I' \in {\cal I}$, 
  and for $I,I' \in {\cal I}$ with $|I| < |I'|$
 there is $e \in I' \setminus I$ such that $I + e \in {\cal I}$, 
 and $\max_{I \in {\cal I}} |I| < + \infty$.
 A member of ${\cal I}$ is called an {\em independent set}. 
 A maximal independent set is called a {\em base}.
 The set of all bases is denoted by ${\cal B}$.
 A matroid can be defined by the base family, and also be written as
 ${\bf M} = (E, {\cal B})$. 
 Bases have the same cardinality $(< +\infty)$, 
 which is called the {\em rank} of ${\bf M}$.
 A {\em loop} is an element $e \in E$ such that no base contains $e$.
 Non-loop elements $e,f \in E$ are said to be {\em parallel} 
 if no base contains both $e$ and $f$. 
 The parallel relation gives rise to 
 an equivalence relation on the set of non-loop elements, 
 and an equivalence class 
 is called a {\em parallel class}.
 If matroid ${\bf M}$ has no loop and no parallel pair, then ${\bf M}$ is called {\em simple}.
 For a subset $E' \subseteq E$ obtained by 
selecting one element from each parallel class, 
we obtain a simple matroid ${\bf M}' = (E', {\cal I}')$ on $E'$, 
where ${\cal I}' := \{ I \in {\cal I} \mid I \subseteq E' \}$.
This matroid ${\bf M}'$ is  called a {\em simplification} of ${\bf M}$.
 The {\em rank function} $\rho:2^E \to \ZZ$ is defined by 
 $\rho(X) := \max\{ |I| \mid I \in {\cal I}: I \subseteq X\}$.
 The {\em closure operator} ${\rm cl}$ is defined by 
 ${\rm cl}(X) = \{e \in E \mid \rho(X+e) = \rho(X) \}$.
 A {\em flat} is a subset $F \subseteq E$ with ${\rm cl} (F) = F$.
 A parallel class is exactly a flat $F$ with $\rho(F) = 1$. 
 The family of all flats becomes a lattice with respect to the inclusion order.
 
 Let us review the relationship between matroids and geometric lattices.
 Let ${\cal L}$ be a geometric lattice with height function $r$. 
 Assume $r(\bar 0) = 0$.
 A subset $I$ of atoms of ${\cal L}$ is called {\em independent} 
 if $r( \bigvee I  ) = |I|$. 
 
 \begin{Thm}[{\cite{Birkhoff}; see \cite[Chapter VI]{Aigner}}]\label{thm:Birkhoff}
 	\begin{itemize}
 		\item[{\rm (1)}] For a geometric lattice ${\cal L}$ with rank $n$,  
 		the pair ${\bf M}_{\cal L}$ of the set of atoms and the family of independent atoms is a simple matroid with rank $n$. 
 		\item[{\rm (2)}] 
 		The family of flats of a matroid ${\bf M}$ with rank $n$ 
 		is a geometric lattice ${\cal L}$ with rank $n$, 
 		where ${\bf M}_{\cal L}$ 
 		is a simplification of ${\bf M}$. 
 	\end{itemize} 
\end{Thm}

\section{Valuated matroids and tropical linear spaces}\label{sec:valuated}

Let ${\bf M} = (E,{\cal B})$ be a matroid with rank $n$.
A {\em valuated matroid} on ${\bf M}$ 
is a function $\omega: {\cal B} \to \RR$ satisfying: 
\begin{itemize}
	\item[(EXC)] For $B, B' \in {\cal B}$ and $e \in B \setminus B'$ there is $e' \in B' \setminus B$ such that
	\begin{equation}
	\omega(B) + \omega(B') \leq \omega(B + e' - e ) + \omega(B' + e - e').
	\end{equation}
\end{itemize}
A valuated matroid $\omega$ is viewed as a function on the set of 
all $n$-element subsets of $E$ by defining $\omega(B) = - \infty$ 
for $B \not \in {\cal B}$. 
A valuated matroid is also written as a pair $(E, \omega)$.
A valuated matroid is called {\em simple} if 
the underlying matroid is a simple matroid.
\begin{Lem}[{\cite{DressTerhalle93}}]\label{lem:parallel}
	Let $(E, \omega)$ be a valuated matroid.
	If $e,f \in E$ are parallel in the underlying matroid, then there is $\alpha \in \RR$ 
	such that $\omega(K+e) = \omega (K+f) + \alpha$
	for every $(n-1)$-element subset $K \subseteq E \setminus \{e,f\}$.
\end{Lem}
Therefore no essential information is lost
when a valuated matroid $(E,\omega)$ is restricted to a simplification 
of the underlying matroid.
The obtained simple valuated matroid $(\tilde E, \tilde \omega)$ 
is called a {\em simplification} of $(E, \omega)$.

For $\omega: {\cal B} \to \RR$ and $x \in \RR^E$, 
define $\omega + x:{\cal B} \to \RR$  by
\begin{equation*}
(\omega+ x)(B) := \omega(B) + \sum_{e \in B} x(e) \quad (B \in {\cal B}).
\end{equation*}
It is easy to see from (EXC) that 
$\omega + x$ is a valuated matroid 
if $\omega$ is a valuated matroid.
Two valuated matroids $\omega$ and $\omega'$ are said to be 
{\em projectively equivalent} if $\omega' = \omega + x$ for some $x \in \RR^E$.

For $\omega: {\cal B} \to \RR$, 
let ${\cal B}_{\omega}$ be 
the set of all bases $B$ that 
attain $\max_{B \in {\cal B}} \omega(B)$.
A direct consequence of (EXC) is as follows.
\begin{Lem}[{\cite{DressWenzel_greedy}; see \cite[Theorem 5.2.7]{MurotaMatrix}}]\label{lem:opt}
	Let $\omega$ be a valuated matroid on $(E, {\cal B})$.
	A base $B \in {\cal B}$ belongs to ${\cal B}_{\omega}$ if and only if
	\begin{equation*}
	\omega (B - e + f) \leq \omega(B)
	\end{equation*}
	for all $e \in B$ and $f \in E \setminus B$ with $B - e + f \in {\cal B}$.
\end{Lem}

One can also observe from (EXC)
that for a valuated matroid $\omega$, the maximizer family  
${\cal B}_{\omega}$ is the base family of a matroid.
Murota~\cite{Murota97} proved that 
this property characterizes valuated matroids when $E$ is finite.
\begin{Lem}[{\cite{Murota97}; see \cite[Theorem 5.2.26]{MurotaMatrix}}]\label{lem:murota}
	Let ${\bf M} = (E,{\cal B})$ be a matroid. 
	An upper-bounded 
	integer-valued function $\omega:{\cal B} \to \ZZ$ is a valuated matroid 
	if and only if 
	$(E, {\cal B}_{\omega+x})$ is a matroid 
	for every bounded integer vector $x \in \ZZ^E$.
\end{Lem}
\begin{proof}[Proof: Reduction to finite case]
	We reduce the proof of the if-part to finite case.
	Consider bases $B,B' \in {\cal B}$. 
	Let $E' := B \cup B'$, and let $(E',\omega')$ be the restriction of $(E,\omega)$. 
	By the upper-boundedness of $\omega$, 
	for $x' \in \ZZ^{E'}$, by choosing a large positive integer $M$ and 
	by defining $x(e) := - M$ $(e \in E \setminus E')$,  
	we can extend $x'$ to bounded vector $x \in \ZZ^E$ so that
	${\cal B}_{\omega + x} ={\cal B}_{\omega'+x'} \subseteq 2^{E'}$.
    Thus the exchange property (EXC) of 
    $\omega$ on $B$ and $B'$ follows from that of~$\omega'$.
\end{proof}
Next we introduce the {\em tropical linear space}~\cite{MurotaTamura01,Speyer08} 
of a  valuated matroid.
Let $\omega$ be an integer-valued valuated matroid on $(E, {\cal B})$.
To deal with a possible infiniteness of $E$, 
we here employ the following definition. 
The {\em tropical linear space} ${\cal T}(\omega)$ of $\omega$ is defined 
as the set of all vectors $x \in \RR^E$ such that matroid 
${\bf M}_{\omega + x} = (E, {\cal B}_{\omega +x})$
has no loop, i.e.,
\begin{equation*}
{\cal T}(\omega) := \{ x \in \RR^E \mid  \mbox{${\bf M}_{\omega + x}$ has no loop}\}.
\end{equation*}
This definition tacitly imposes 
that the maximum of $\omega+x$ for $x \in {\cal T}(\omega)$ 
is attained by some $B \in {\cal B}$. 
According to the definition in \cite{MurotaTamura01,Speyer08},  
the tropical linear space 
is the set of all points $x \in \RR^E$ satisfying:
\begin{itemize}
	\item[(TW)] For any $(n+1)$-element subset $C \subseteq E$, 
	the maximum of $\omega(C - f ) - x(f)$ over all $f \in C$ with $C - f \in {\cal B}$ is attained at least twice.
\end{itemize}
(In the definition of \cite{MurotaTamura01}, the sign of $x$ is opposite.)
Speyer~\cite[Proposition 2.3]{Speyer08} proved that the two definitions are 
equivalent when $E$ is finite.
Our infinite setting  needs a little care; 
we prove a slightly modified equivalence in Lemma~\ref{lem:trop} below.

In the literature, 
the tropical linear space is referred to as its projection ${\cal T}(\omega)/ \RR{\bf 1}$, 
since $x \in {\cal T}(\omega)$ implies 
$x+ \RR {\bf 1} \subseteq {\cal T}(\omega)$.
Earlier than \cite{MurotaTamura01,Speyer08}, 
Dress and Terhalle~\cite{DressTerhalle93,DressTerhalle98} introduced
the concept of 
the {\em tight span} ${\cal TS}(\omega)$ of $\omega$, which is defined by 
\begin{equation*}
{\cal TS}(\omega) := \left\{ p \in \RR^E \bigmid 
p(e) = \max_{B \in {\cal B}: e \in B} \{ \omega(B) - \sum_{f \in B \setminus \{e\}} p(f)\} \quad (e \in E) \right\}.
\end{equation*}
Observe that ${\cal TS}(\omega)$ is the set of representatives of the negative of ${\cal T}(\omega)/ \RR{\bf 1}$. More precisely, it holds
\begin{equation}\label{eqn:TS}
{\cal TS}(\omega)  = - \{x \in {\cal T}(\omega) \mid \max_{B \in {\cal B}} (\omega + x)(B) = 0 \}.
\end{equation}

Dress and Terhalle~\cite{DressTerhalle93,DressTerhalle98} introduced 
an ultrametric metrization of 
the ground set $E$ of a valuated matroid $\omega$, which we explain below.
Let us recall the notion of an ultrametric. 
An {\em ultrametric} on a set $X$
is a metric $d: X \times X \to \RR_+$ satisfying the ultrametric inequality
\begin{equation}\label{eqn:ultrametric}
d(x,y) \leq \max \{d(x,z), d(z,y)\} \quad (x,y,z \in X).
\end{equation}
For $p \in {\cal TS}(\omega)$,
define $D_p: E \times E \to \RR$ by 
\begin{equation*}
D_p(e,f) := \left\{
\begin{array}{ll}
\exp (\max \{ (\omega - p) (B) \mid B \in {\cal B}: \{e,f\} \subseteq B \}) & {\rm if}\ e \neq f,\\
0 & {\rm if}\ e = f
\end{array}\right.
 \quad (e,f \in E).
\end{equation*}
\begin{Prop}[\cite{DressTerhalle93}]
	Let $(E,\omega)$ be a simple valuated matroid. 
	For $p \in {\cal TS}(\omega)$, we have the following:
	\begin{itemize}
		\item[{\rm (1)}] $D_{p}$ is an ultrametric.
		\item[{\rm (2)}] For $q \in {\cal TS}(\omega)$, 
		it holds $\alpha D_p \leq D_q \leq \beta D_p$ for some $\alpha,\beta > 0$.
	\end{itemize} 
\end{Prop}
A simple valuated matroid $(E,\omega)$ is called {\em complete} 
if the metric space $(E, D_p)$ is a complete metric space.
By the property (2)
the convergence property is independent of the choice of $p \in {\cal TS}(\omega)$.
A {\em completion} of a valuated matroid $(E, \omega)$ 
is a complete valuated matroid $(\bar E,\bar \omega)$ with the properties 
that 
$\bar E$ contains $E$ as a dense subset, and 
$\omega$ is equal to the restriction of $\bar \omega$ to $n$-element subsets in $E$. 
\begin{Thm}[\cite{DressTerhalle93}]
	For a simple valuated matroid $(E, \omega)$, 
	there is a (unique) completion $(\bar E, \bar \omega)$ of $(E, \omega)$.
\end{Thm}

The construction of a completion of valuated matroid $(E,\omega)$ 
is analogous to (and generalizes) 
that of $p$-adic numbers from rational numbers:
Consider the set $\bar E$ of 
all Cauchy sequences $(x_i)$, relative to $D_p$, 
modulo the equivalence relation $\sim$ defined by 
$(x_i) \sim (y_i) \Leftrightarrow \lim_{i \rightarrow \infty} D_p(x_i,y_i) = 0$.
Regard $E$ as a subset of $\bar E$ by associating $x \in E$ with 
a Cauchy sequence converging to $x$, and
extend $D_p$ to $\bar E \times \bar E \to \RR$ by
$D_p((x_i),(y_i)) := \lim _{i \rightarrow \infty} D_p(x_i,y_i)$. 
Then $\bar E$ is a complete metric space containing $E$ as a dense subset.
Accordingly, $\omega$ is extended to $\bar \omega$ by 
\begin{equation*}
\bar \omega(B) := \lim_{i \rightarrow \infty} \omega(B_i),
\end{equation*} 
where $B_i \subseteq E$ consists of $n$ elements each converging to an element of $B \subseteq \bar E$.
By a completion of nonsimple valuated matroid $(E, \omega)$ 
we mean a completion of a simplification of $(E, \omega)$. 
In Section~\ref{subsec:v->u}, 
we give a natural interpretation of 
this completion process via
a uniform semimodular lattice. 
 
The rest of this section is to 
give some basic properties of the tropical linear space
${\cal T}(\omega)$. Let $(E, \omega)$ be 
an integer-valued valuated matroid on 
underlying matroid ${\bf M} = (E,{\cal B})$ of rank $n$.
We suppose that ${\cal T}(\omega)$ is endowed with the vector order $\leq$.
\begin{Lem}\label{lem:projection}
	Let $(\tilde E, \tilde \omega)$ be a simplification of $(E, \omega)$.
	Then the projection $x \mapsto x|_{\tilde E}$ is 
	an order-preserving bijection from ${\cal T}(\omega)$ to ${\cal T}(\tilde \omega)$.
\end{Lem}
\begin{proof}
	Let $e,f \in E$ be a parallel pair with $\omega (K+e) = \omega(K+f) + \alpha$ 
	for every $(n-1)$-element subset $K \subseteq E \setminus \{e,f\}$; see Lemma~\ref{lem:parallel}. 
	Let $x \in {\cal T}(\omega)$. 
	There is $B \in {\cal B}_{\omega+x}$ containing $e$.
	By $(\omega+x)(B) \geq (\omega + x)(B - e +f)$ and 
	$\omega (B) = \omega(B- e+ f) + \alpha$, 
	we have $x(e) \geq x(f) - \alpha$. Similarly, 
	by taking $B' \in {\cal B}_{\omega+x}$ with $f \in B'$, 
	we obtain $x(e) = x(f) - \alpha$; 
	then  $(\omega+x)(B) = (\omega + x)(B - e +f)$ holds for the above $B$.
	In particular, if 
	$B \in {\cal B}_{\omega+x}$ contains $e$, 
	then $B - e + f \in {\cal B}_{\omega+x}$.
	From this, we see that ${\cal B}_{\tilde \omega + x|_{\tilde E}}$ is a subset of 
	${\cal B}_{\omega + x}$, and ${\bf M}_{\tilde \omega + x|_{\tilde E}}$ has no loop.
	Thus the projection $x \mapsto x|_{\tilde E}$ is 
	an order-preserving map from ${\cal T}(\omega)$ to~${\cal T}(\tilde \omega)$.
	By $x(e) + \alpha = x(f)$, where $\alpha$ is independent of $x$, 
	we see that the projection is a bijection.
\end{proof}
For $x \in \RR^E$, let $\lfloor x \rfloor \in \ZZ^E$ 
denote the vector obtained from $x$ 
by rounding down each fractional component of $x$, i.e. 
$\lfloor x \rfloor (e) := \lfloor x(e) \rfloor$ for $e \in E$.

\begin{Lem}\label{lem:chain_of_flats}
	For $x \in {\cal T}(\omega)$, we have the following:
	\begin{itemize}
		\item[{\rm (1)}] 
		$\lfloor x \rfloor \in {\cal T}(\omega)$. 
		\item[{\rm (2)}] There are a chain 
		$\emptyset \neq F_1 \subset F_2 \subset \cdots \subset F_n = E$ of flats in ${\bf M}_{\omega + \lfloor x \rfloor}$ and coefficients $\lambda_i \geq 0$ such that  $\sum_{i=1}^n \lambda_i < 1$ and
		$x = \lfloor x \rfloor + \sum_{i=1}^{n} 
		\lambda_i {\bf 1}_{F_i}$.
	\end{itemize} 
\end{Lem}
\begin{proof}
	(1). Let $B \in {\cal B}_{\omega+x}$. 
	By Lemma~\ref{lem:opt}, 
	we have $(\omega + x) (B + e - f) \leq (\omega + x)(B)$ 
	for all $e \in E \setminus B$ and $f \in B$.
	From this, we have 
	\begin{equation*}
	(\omega + \lfloor x \rfloor) 
	(B + e - f) + \varDelta(e) - \varDelta(f) \leq (\omega + \lfloor x \rfloor)(B),
	\end{equation*}
	where $\varDelta(g) := x(g) - \lfloor x(g) \rfloor \in [0,1)$ for $g \in E$.
	Since $\varDelta(e) - \varDelta(f) > - 1$ and $\omega$ is integer-valued, 
	we have $(\omega + \lfloor x \rfloor) 
	(B + e - f)  \leq (\omega + \lfloor x \rfloor)(B)$. By Lemma~\ref{lem:opt} again, 
	we have $B \in {\cal B}_{\omega+ \lfloor x \rfloor}$.
	Hence ${\cal B}_{\omega + x} \subseteq {\cal B}_{\omega + \lfloor x \rfloor}$.
	Therefore ${\bf M}_{\omega + \lfloor x \rfloor}$ has no loop. 
	
	(2). It suffices to show: 
	If $x \in {\cal T}(\omega)$ and $\alpha \in [0,1)$, then
	$F_{\alpha} := \{e \in E \mid x(e) - \lfloor x(e) \rfloor \geq \alpha\}$ is a flat of matroid ${\bf M}_{\omega + \lfloor x \rfloor}$. 
	By ${\cal B}_{\omega + x} \subseteq {\cal B}_{\omega + \lfloor x \rfloor}$ shown above, one can see that ${\cal B}_{\omega+ x}$ 
	is the maximizer family of 
	linear objective function $B \mapsto (x -  \lfloor x \rfloor)(B)$ 
	over ${\cal B}_{\omega + \lfloor x \rfloor}$.
	Suppose to the contrary that $e \in {\rm cl} (F_{\alpha}) \setminus F_\alpha$ exists.
	Take a base $B \in {\cal B}_{\omega+ x}$ containing $e$.
	Then ${\rm cl} (B - e) \not \supseteq F_{\alpha}$ 
	since otherwise $e \not \in {\rm cl} (B - e) = {\rm cl} ({\rm cl} (B - e)) \supseteq {\rm cl} (F_{\alpha})\ni e$.
	Thus we can choose $f \in F_{\alpha}$ such that $B + f - e \in {\cal B}_{\omega + \lfloor x \rfloor}$.
	But the above linear objective function increases strictly.  
	This is a contradiction.
\end{proof}

 \begin{Lem}\label{lem:trop}
 	A vector $x \in \RR^E$ belongs to ${\cal T}(\omega)$ if and only if 
 	the maximum of $\omega+x$ over ${\cal B}$ is attained and $x$ satisfies {\rm (TW)}. 
 \end{Lem}
 \begin{proof}
 	(If part). Consider $B \in {\cal B}_{\omega + x}$ and $e \in E \setminus B$.
 	Then $\max_{f \in B+e: B + e - f \in {\cal B}} \omega(B + e - f) - x(f)$ is attained by $f = e$ (Lemma~\ref{lem:opt}), 
 	and $f  \neq e$ by (TW).
 	This means that $B + e- f \in {\cal B}_{\omega + x}$.
 	Thus ${\bf M}_{\omega+x}$ is loop-free. 
 	 
 	(Only if part). By Lemma~\ref{lem:chain_of_flats}~(2) and $|F_i \cap B| \in \{0,1,2,\ldots,n\}$, 
    it holds $\{ (\omega + x)(B) \mid B \in {\cal B}\} \subseteq \ZZ + U$
    for a finite set $U = \{\sum_{i=1}^{n} \lambda_i k_i \mid 0 \leq k_i \leq n\}$. 
    Consequently the maximum of $\omega + x + \alpha {\bf 1}_{F}$ is attained  
    for all $\alpha \geq 0$ and $F \subseteq E$.
    The rest is precisely the same as in the proof of \cite[Proposition 2.3]{Speyer08}.
    Consider an arbitrary $n+1$ element subset $C$.
    As $\alpha \geq 0$ increases,
    the maximizer family  
    ${\cal B}_{\omega + x + \alpha {\bf 1}_{C}}$ changes finitely many times.
    Also, for large $\alpha \geq 0$, ${\cal B}_{\omega + x + \alpha {\bf 1}_{C}}$
    consists only of bases $B \in {\cal B}$ with $B \subseteq C$.
    We show that each $e \in C$ is not a loop in ${\cal B}_{\omega + x + \alpha {\bf 1}_{C}}$ for $\alpha \geq 0$. 
    For small $\epsilon > 0$, any base $B$ 
    in ${\cal B}_{\omega + x + \alpha {\bf 1}_{C}}$ with maximal $C \cap B$ 
    is also a base in  
    ${\cal B}_{\omega + x + (\alpha+ \epsilon) {\bf 1}_{C}}$; see below.
    Since each $e \in C$ is not a loop in ${\bf M}_{\omega + x}$, 
    so is in ${\bf M}_{\omega + x + \alpha {\bf 1}_{C}}$.  
    Thus, for large $\alpha>0$, the maximum of $\omega + x + \alpha {\bf 1}_{C}$
    must be attained by at least two bases in $C$, which implies (TW).   
 \end{proof}
 In the last step of the proof, we use the following lemma:
 \begin{Lem}[\cite{Speyer08}]\label{lem:integer-valued}
 	Let $x \in {\cal T}(\omega)$ and $F \subseteq E$.
 	Any base $B \in {\cal B}_{\omega+x}$ 
 	with maximal $B \cap F$  
 	belongs to ${\cal B}_{\omega + x + \alpha {\bf 1}_F}$
 	for sufficiently small $\alpha > 0$.
 	If $\omega$ and $x$ are integer-valued, 
 	then we can take $\alpha = 1$.
 \end{Lem}
 \begin{proof}
 We only show the case where $\omega$ and $x$ are integer-valued;
 the proof of the non-integral case is essentially the same.
 We can assume that $x = {\bf 0}$.
 Consider a base $B \in {\cal B}_{\omega}$ 
 with maximal $B \cap F$.  
By Lemma~\ref{lem:opt}, 
it suffices to show that $\omega(B) + |B \cap F| \geq 
\omega (B - e + f) + |(B - e + f) \cap F|$ for $e \in B$ 
and $f \in E \setminus B$ with  $B - e + f \in {\cal B}$.
By $\omega(B) \geq \omega(B-e+f)$, 
if $f \not \in F$ or $e \in F$, 
then this obviously holds.
Suppose that $f \in F$ and $e \not \in F$.
Then  $|(B - e + f) \cap F| = 1 + |B \cap F|$.
By the maximality, $B-e+f$ does not belong to 
${\cal B}_{\omega}$, implying $\omega(B - e+f) \leq \omega(B) - 1$.
Thus  $\omega(B) + |B \cap F| \geq 
\omega (B - e + f) + |(B - e + f) \cap F|$.
 \end{proof}

The tropical linear space enjoys 
a tropical version of convexity
introduced by Develin-Sturmfels~\cite{DevelinSturmfels}. 
A subset $Q \subseteq \RR^E$ is said to 
be {\em tropically convex}~\cite{DevelinSturmfels}
if $\min (x + \alpha {\bf 1}, y + \beta {\bf 1}) \in Q$ for all $x,y \in Q$ and $\alpha,\beta \in \RR$. An equivalent condition for the tropical convexity 
consists of (TC$_\wedge$) and (TC$_{+{\bf 1}}$) below:
\begin{itemize}
	\item[(TC$_\wedge$)] $\min(x,y) \in Q$ for all $x,y \in Q$.
	\item[(TC$_{+{\bf 1}}$)] $x+ \alpha {\bf 1} \in Q$ for all $x \in Q$, $\alpha \in \RR$. 
\end{itemize}
These two properties of ${\cal T}(\omega)$ 
were recognized  
by Murota-Tamura~\cite{MurotaTamura01} (in finite case).
\begin{Lem}[{\cite[Theorem 3.4]{MurotaTamura01}; 
		see also \cite[Proposition 2.14]{Hampe15}}]\label{lem:trop_convexity}
	 The tropical linear space ${\cal T}(\omega)$ is tropically convex.
\end{Lem}
\begin{proof}
We show that ${\cal T}(\omega)$ satisfies (TC$_\wedge$), while 
(TC$_{+{\bf 1}}$) is obvious.
Let $x,y \in {\cal T}(\omega)$. As in the proof of Lemma~\ref{lem:trop},
we see from Lemma~\ref{lem:chain_of_flats}~(2) that the image 
$\{ (\omega + x \wedge y)(B) \mid B \in {\cal B} \}$ of $\omega + x \wedge y$ 
is discrete in $\RR$.
Hence the maximum of $\omega + x \wedge y$ is attained by some base.
Let $C$ be an $(n+1)$-element subset of $E$.
We may assume that $\max_f \omega(C - f ) - x(f) \geq \max_f \omega(C - f ) - y(f)$. Necessarily $\max_f \omega(C - f ) - (x \wedge y) (f) = \max_f \omega(C - f ) - x(f)$. By (TW) and Lemma~\ref{lem:trop}, we can choose distinct $e,e' \in C$ that attain  
$\max_f \omega(C - f ) - x(f)$. Necessarily $x(e) = (x \wedge y)(e)$ and 
$x(e') = (x \wedge y)(e')$. Thus $e,e'$ attain $\max_f \omega(C - f ) - (x \wedge y) (f)$. By Lemma~\ref{lem:trop}, we have $x \wedge y \in {\cal T}(\omega)$. 
\end{proof}
By this property, ${\cal T}(\omega) \cap \ZZ^E$ 
becomes a lattice with respect to the vector order $\leq$.  
In the next section, 
we characterize this lattice ${\cal T}(\omega) \cap \ZZ^E$.

\section{Uniform semimodular lattices}\label{sec:uniform}

The {\em ascending operator} of a lattice ${\cal L}$ is a map $(\cdot)^+:{\cal L} \to {\cal L}$ defined by
\begin{equation*}
(x)^+ := \bigvee \{y \in {\cal L} \mid \mbox{$y$ covers $x$}\}.
\end{equation*} 
A {\em uniform semimodular lattice} ${\cal L}$ is a semimodular lattice such that the ascending operator $(\cdot)^+$ is defined, 
and is an automorphism on ${\cal L}$. 
If, in addition, ${\cal L}$ is a modular lattice, then ${\cal L}$ 
is called a {\em uniform modular lattice}.
The condition for $(\cdot)^+$ is viewed as
a lattice-theoretic analogue of condition (TC$_{+{\bf 1}}$).
The simplest but important example of  
a uniform (semi)modular lattice is $\ZZ^m$:
\begin{Ex}
	View $\ZZ^m$ as a poset with vector order $\leq$.
	Then $\ZZ^m$ is a lattice, where the join $x \vee y$ and meet $x \wedge y$ 
	are $\max(x,y)$ and $\min(x,y)$, respectively.
	The component sum $x \mapsto \sum_{i=1}^m x_i$ is a height function 
	satisfying the semimodularity inequality (\ref{eqn:semimo}) (in equality).
	Therefore $\ZZ^m$ is a (semi)modular lattice.
	Observe that the ascending operator is equal to $x \mapsto x + {\bf 1}$, which is obviously an automorphism. 
	Thus $\ZZ^m$ is a uniform (semi)modular lattice (with uniform-rank $m$).
\end{Ex}
\begin{Ex}\label{ex:Z^E,n}
	Let $E$ be a finite set with $|E| = m$.
	Consider the poset $\ZZ^E (\simeq \ZZ^m)$ by vector order $\leq$.
	For positive integer $n$ with $n \leq m$, 
	let $\ZZ^{E,n}$ denote the subposet of  $\ZZ^E$ consisting of all $x$
	such that the minimum of $x(e)$ over $e \in E$ is attained 
	by at least $m - n + 1$ elements.
	If $n= m$, then $\ZZ^{E,n}$ is equal to 
	the above uniform modular lattice $\ZZ^E \simeq \ZZ^m$.
	By using notation $\Argmin x := \{e \in E \mid x(e) = \min_{f \in E} x(f)\}$,  
	$\ZZ^{E,n}$ is written as
	\[
	\ZZ^{E,n} = \{ x \in \ZZ^E \mid |\Argmin x| \geq m - n +1 \}.
	\]
	One can see from (TW) that $\ZZ^{E,n}$ is 
	the set of integer points of 
	the tropical linear space of the trivial valuation ($\omega = 0$) of the uniform matroid, or
	the {\em Bergman fan} of the uniform matroid (of rank $n$); 
	see  Section~\ref{sec:example}.
	
	It might be instructive to verify from the definition 
	that $\ZZ^{E,n}$ is a uniform semimodular lattice. 
	Since $x,y \in \ZZ^{E,n}$ 
	implies $\min (x,y) \in \ZZ^{E,n}$, 
	the meet $x \wedge y$ is equal to $\min (x,y)$.
	Then $\ZZ^{E,n}$ becomes a lattice in which the join $x \vee y$ 
	is given by $x \vee y = \bigwedge \{z \in \ZZ^{E,n} \mid z \geq \max (x,y)\}$.
	We next show the semimodularity. 
	For $x,y \in \ZZ^{E,n}$,
	if $|\Argmin x| = m - n +1$, 
	then $y$ covers $x$ if and only 
	if $y = x + {\bf 1}_{\Argmin x}$ 
	or $y = x + {\bf 1}_{e}$ for some $e \in E \setminus \Argmin x$.
	If $|\Argmin x| > m -n +1$, then $y$ covers $x$ if and only 
	if $y = x + {\bf 1}_{e}$ for some $e \in E$.
	From this, one can verify the condition of Lemma~\ref{lem:semimodular}~(2).
	For $x,y$ covering $z (= x \wedge y)$, if $|\Argmin z| = m -n + 2$, 
	$x = z + 1_{e}$ and $y = z + 1_{f}$ for distinct $e, f \in \Argmin z$, 
	then $x \vee y$ is equal to $z + {\bf 1}_{\Argmin z}$, which covers $x,y$.
	For other cases, $x \vee y$ is equal to $\max (x,y)$, 
	which obviously covers $x,y$.
	Hence $\ZZ^{E,n}$ is a semimodular lattice.
	Also $(x)^+$ is given by $x \mapsto x + {\bf 1}$,
	 which is obviously an automorphism.  
	The uniform-rank is equal to $n$, since 
	${\bf 0}, {\bf 1}_{\{e_1\}}, {\bf 1}_{\{e_1,e_2\}},\ldots, {\bf 1}_{\{e_1,\ldots, e_{n-1}\}}, {\bf 1}$ is a maximal chain of $[{\bf 0}, {\bf 1}]$.
	
\end{Ex}

\subsection{Basic concepts and properties}
In this section, we introduce basic concepts on uniform semimodular lattices
and prove some of basic properties, 
which will be a basis of our cryptomorphic equivalence to valuated matroids.
Some of them were introduced and proved in \cite{HH18a} for uniform modular lattices.

Let ${\cal L}$ be a uniform semimodular lattice, and let $r$ denote a height function of ${\cal L}$.
\begin{Lem}\label{lem:u-rank}
	For $x,y \in {\cal L}$, 
	the intervals $[x, (x)^+]$ and $[y, (y)^+]$ are 
	geometric lattices of the same rank.
\end{Lem}
\begin{proof}
	The semimodularity of $[x,(x)^+]$ is obvious.
	We show that every element in $[x,(x)^+]$ is 
	the join of a subset of atoms ($=$ elements covering $x$).
	Take arbitrary $y \in [x,(x)^+]$.
	Since $(\cdot)^+$ is an automorphism, 
	we can take $z \in {\cal L}$ with $(z)^+ = y$ and $z \preceq x$.
	By definition, $y$ is the join of atoms in $[z,y]$, i.e., 
	the join of elements $z_1,z_2,\ldots,z_k$ covering $z$. 
	Consider elements $x \vee z_1, x \vee z_2, \ldots, x \vee z_k$, by the semimodularity, each of which equals $x$ or covers $x$.
	Also their join is equal to $y$. 
	This means that $y$ is represented as the join of atoms in $[x,(x)^+]$.
	Hence $[x, (x)^+]$ is a geometric lattice.

	We show that $[x, (x)^+]$ and $[y, (y)^+]$ have the same rank.
	It suffices to consider the case where $y$ covers $x$.
	Since $(\cdot)^+$ is an automorphism, 
	$(y)^+$ covers $(x)^+$.
	Therefore we have $1 + r[y,(y)^+] = r[x, (y)^+] = r[x,(x)^+] + 1$ (by (JD)), which implies $r[x,(x)^+] = r[y, (y)^+]$.
\end{proof}
The {\em uniform-rank} of ${\cal L}$ is defined as the rank $r[x, (x)^+]$
of interval $[x,(x)^+]$ for $x \in {\cal L}$.
We next study the inverse $(\cdot)^-$ of the ascending operator $(\cdot)^+$.
\begin{Lem}\label{lem:inverse}
	The inverse $(\cdot)^-$ of $(\cdot)^+$ is given by
	\begin{equation}\label{eqn:x^-}
	(x)^- = \bigwedge \{ y \in {\cal L} \mid \mbox{$y$ is covered by $x$}\} \quad (x \in {\cal L}).
	\end{equation}
\end{Lem}
\begin{proof}
	Suppose that $y \in {\cal L}$ is covered by $(x)^+$.
	Since $(\cdot)^+$ is an automorphism, 
	there is $y' \in {\cal L}$ such that $(y')^+ = y$.
	Also $x$ covers $y'$,  
	which implies $x \preceq (y')^+= y$ by the definition of $(\cdot)^+$.
	Namely $y$ belongs to $[x, (x)^+]$.
	Now $x$ is also the meet of all hyperplanes 
	of geometric lattice $[x, (x)^+]$ (Lemma~\ref{lem:hyperplane}~(1)).
	By the above argument, they are exactly elements covered by $(x)^+$ in ${\cal L}$. 
	This means that the right hand side of (\ref{eqn:x^-}) exists, 
	and equals $(x)^-$.  
\end{proof}

For $x \in {\cal L}$ and $k \in \ZZ$, define $(x)^{+k}$ by
\begin{equation*}
(x)^{+k} := \left\{
\begin{array}{ll}
x  & {\rm if}\ k= 0, \\
((x)^{+(k-1)})^{+} & {\rm if}\ k > 0, \\
((x)^{+(k+1)})^{-} & {\rm if}\ k < 0.  
\end{array}\right.
\end{equation*}
For $k > 0$, we denote $(x)^{+(-k)}$ by $(x)^{-k}$.
\begin{Lem}\label{lem:preceq}
	For $x,y \in {\cal L}$, 
	there is $k \geq 0$ 
	such that $x \preceq (y)^{+k}$.
\end{Lem}
\begin{proof}
	We may assume that $x \not \preceq y$.
	Hence $x \succ x \wedge y$.
	Choose an atom $a$ in $[x \wedge y, x]$.
	Then $a \wedge y = x \wedge y$. By semimodularity,	 
	$a \vee y$ is an atom in $[y, (y)^+]$, and
	$x \wedge y \prec a \preceq x \wedge (y)^+$.
	Thus, for $k \geq r[x \wedge y,x]$, it holds $x \wedge (y)^{+k} = x$, 
	implying $x \preceq (y)^{+k}$.
\end{proof}

\subsubsection{Segments and rays}\label{subsec:segments}
A {\em segment} is a chain $e^0 \prec e^1 \prec \cdots \prec e^s$
such that $e^\ell$ covers $e^{\ell-1}$ for $\ell=1,2,\ldots,s$, 
and $e^{\ell+1} \not \in [e^{\ell-1}, (e^{\ell-1})^+] (\ni e^\ell)$ for $\ell=1,2,\ldots,s-1$.
A {\em ray} is an infinite chain $e^0 \prec e^1 \prec \cdots$
such that $e^0 \prec e^1 \prec \cdots \prec e^\ell$ is a segment for all $\ell$.
\begin{Ex}\label{ex:ray}
	Consider the case of ${\cal L} = \ZZ^n$.
	Then a ray is precisely a chain
	\begin{equation}\label{eqn:ray_Z^n}
	x \prec x+ {\bf 1}_{i} \prec x+ 2{\bf 1}_{i} \prec x+ 3{\bf 1}_{i} \prec \cdots 
	\end{equation}
	for some $x \in \ZZ^n$ and $i \in \{1,2,\ldots,n\}$, 
	where ${\bf 1}_i$ denote the $i$-th unit vector.
	
	The case of ${\cal L} = \ZZ^{E,n}$ is similar. 
	But, for $x \in \ZZ^{E,n}$ 
	with $|\Argmin x| = m - n +1$ and $e \in \Argmin x$,  the chain 
	\begin{equation}\label{eqn:ray_Z^E,n}
	x \prec x + {\bf 1}_{ \Argmin x} \prec x + 2 {\bf 1}_{ \Argmin x} \prec \cdots \prec y \prec  y + {\bf 1}_{e} \prec y + 2 {\bf 1}_{e} \prec \cdots
	\end{equation}
	is also a ray, where $y: = x + k {\bf 1}_{\Argmin x}$ for the difference $k$ between 
	the minimum of $x$ and the second minimum.
	Other rays in $\ZZ^{E,n}$ are of form (\ref{eqn:ray_Z^n}).
\end{Ex}

The following characterization of segments was suggested by K. Hayashi.
\begin{Lem}\label{lem:straight}
	A chain $x = e^0 \prec e^1 \prec \cdots \prec e^s = y$ is a segment if and only 
	if $[x,y] = \{ e^0, e^1, \ldots, e^s \}$.
\end{Lem}
\begin{proof}
	(If part). 
	Suppose to the contrary that $e^{\ell+1} \in [e^{\ell-1},(e^{\ell-1})^+]$.
	Then there is an atom $a$ in $[e^{\ell-1},(e^{\ell-1})^+]$ 
	such that $e^{\ell+1} = a \vee e^{\ell}$ (by Lemma~\ref{lem:hyperplane}~(2)).
	This implies that $a \in [e^{\ell-1},e^{\ell+1}]$, 
	which contradicts $[e^{\ell-1},e^{\ell+1}] = \{e^{\ell-1},e^\ell,e^{\ell+1}\}$.
	
	(Only if part). We use the induction on the length $s$; 
	the case of $s=1$ is obvious.
	Suppose that $[x,e^{s-1}] = \{ e^0, e^1, \ldots, e^{s-1}\}$, and suppose to the contrary that $[x,y]$ properly contains $\{ e^0, e^1, \ldots, e^{s}\}$.
	Then (by induction applied to $\{e^1,e^2,\ldots,e^{s-1},e^s\}$) there 
	is an atom $a$ of $[x,y]$ not belonging to $\{ e^0, e^1, \ldots, e^{s}\}$.
	In particular, $a \not \preceq e^{s-1}$.
	By semimodularity,  $a \vee e^{s-1}$ covers $e^{s-1}$, and is equal to $e^s$. 
	Consider $e^{s-2} \vee a$, which covers $e^{s-2}$ and 
	is not equal to $e^{s-1}$ (by $a \not \preceq e^{s-1}$).
	The join $(e^{s-2} \vee a) \vee e^{s-1}$ is equal to $e^{s}$.
	However this contradicts $e^s \not \in [e^{s-2}, (e^{s-2})^+]$.
\end{proof}
A ray (or segment) $e^0 \prec e^1 \prec \cdots$ with $x = e^0$ 
is called an {\em $x$-ray} (or {\em $x$-segment}).
\begin{Lem}\label{lem:segment}
	Let $x = e^0 \prec e^1 \prec \cdots \prec e^s$ be an $x$-segment.
	For $p \succeq x$ with $p \wedge e^1 = x$, 
	 chain  $p = p \vee e^0 \prec p \vee e^1 \prec \cdots \prec p \vee e^{s}$
	is a $p$-segment.
\end{Lem}
\begin{proof}
	It suffices to consider the case where $p$ covers $x$.
	By $p \neq e^1$ and $[e^0,e^s] = \{e^0,e^1,\ldots,e^s\}$ by Lemma~\ref{lem:straight},
	it holds $p \not \preceq e^s$.
	Then, by semimodularity, $(p,e^s)$ is a modular pair.
	Consequently, $p \vee e^{\ell+1}$ covers $p \vee e^{\ell}$ and $e^{\ell+1}$.
	Let $f^{\ell} := p \vee e^{\ell}$. 
	We show $(f^{\ell})^+ = (f^{\ell-1})^+ \vee f^{\ell+1}$, 
	which implies $f^{\ell+1} \not \in [f^{\ell-1}, (f^{\ell-1})^+]$. 
	By $f^{\ell-1} \vee e^\ell = f^\ell$, we have $(f^{\ell-1})^+ \vee (e^\ell)^+ = (f^\ell)^+$.
	By $(e^\ell)^+ = (e^{\ell-1})^+ \vee e^{\ell+1}$, 
	we have $(f^\ell)^+ = (f^{\ell-1})^+ \vee (e^{\ell-1})^+ \vee e^{\ell+1} = (f^{\ell-1})^+ \vee e^{\ell+1} = (f^{\ell-1})^+ \vee f^\ell \vee e^{\ell+1} = (f^{\ell-1})^+  \vee f^{\ell+1}$, as required.
\end{proof}

For $x \in {\cal L}$, let $r_x$ be a height function (on $[x,(x)^+]$) 
defined by $r_x(y) = r(y) - r(x)$.
A set of $x$-rays $(e_i^\ell)$ $(i=1,2,\ldots,k)$
is said to be {\em independent} if
$r_x(e_1^1 \vee e_2^1 \vee \cdots \vee e_k^1) = k$,  
or equivalently if $e_i^1 \wedge (\bigvee_{j\neq i} e_j^1) = x$ for each $i$.
\begin{Prop}\label{prop:generated}
	The sublattice generated by an independent set of $k$ $x$-rays
	is isomorphic to $\ZZ^k_+$, where the isomorphism is given by
	\begin{equation*}
	\ZZ^k_+ \ni (z_1,z_2,\ldots,z_k) \mapsto e_1^{z_1} \vee e_2^{z_2} \vee \cdots \vee e_k^{z_k}. 
	\end{equation*}
\end{Prop}
\begin{proof} Suppose that 
	$x$-rays $(e_i^\ell)$ $(i=1,2,\ldots,k)$ are independent.
	We first show:
	\begin{Clm}
		For $z \in \ZZ^k_+$, we have the following.
		\begin{itemize}
		\item[(1)] $r_x(e_1^{z_1} \vee e_2^{z_2} \vee \cdots \vee e_k^{z_k}) = \sum_{i=1}^k z_i$.
		\item[(2)] $e_j^{z_j} \wedge (\bigvee_{i:i\neq j} e_i^{z_i}) = x$ for $j \in \{1,2,\ldots,k\}$.
		\end{itemize}
	\end{Clm}
	 \begin{proof}
	 	(1). 
	 	We prove the claim by induction on $k$; the case of $k=1$ is obvious.
	 	From Lemma~\ref{lem:straight} and the independence of $(e_i^\ell)$, 
	 	we have $e_j^1 \wedge e_k^{z_k} = x$ for $j=1,2,\ldots,k-1$.
	 	By Lemma~\ref{lem:segment}, $(e_j^\ell \vee e_k^{z_k})$ $(j=1,2,\ldots,k-1)$
	 	are $e_k^{z_k}$-segments.
	 	We next show that they are independent.
	 	Indeed, $e_k^2$ covers $e_k^1$ 
	 	and $e_k^2 \not \preceq e_1^1 \vee e_2^1 \vee \cdots \vee e_k^1$ 
	 	(otherwise $e_k^2 \in [e_k^0,(e_k^0)^+]$).
	 	Thus, by semimodularity,  
	 	$r_{e_k^2} (e_1^1 \vee e_2^1 \vee \cdots \vee e_{k-1}^1 \vee e_k^2) = r_{e_k^1}(e_1^1 \vee e_2^1 \vee \cdots \vee e_{k}^1) = k-1$, and 
	 	$e_j^1 \vee e_k^1$ $(j=1,2,\ldots,k-1)$ are independent in $[e_k^1,(e_k^1)^+]$.
	 	Repeating this, we see that $e_j^1 \vee e_k^{z_k}$ $(j=1,2,\ldots,k-1)$   
	 	are independent in $[e_k^{z_k},(e_k^{z_k})^+]$.
	 	By induction, we have 
	 	$r_x(e_1^{z_1} \vee e_2^{z_2} \vee \cdots \vee e_k^{z_k}) = r_{e_k^{z_k}}(e_1^{z_1} \vee e_2^{z_2} \vee \cdots \vee e_k^{z_k}) + z_k = \sum_{i=1}^k z_i$, as required.
	 	
	 	(2). From (1) and semimodularity (\ref{eqn:semimo}), 
	 	we have $\sum_{i: i \neq j} z_i + z_j = r_x(\bigvee_{i:i\neq j} e_i^{z_i}) + r_x(e_j^{z_j}) \geq r_x(e_1^{z_1} \vee e_2^{z_2} \vee \cdots \vee e_k^{z_k}) + r_x(e_j^{z_j} \wedge (\bigvee_{i:i\neq j} e_i^{z_i})) \geq \sum_{i}z_i$.
	 	Thus $r_x(e_j^{z_j} \wedge (\bigvee_{i:i\neq j} e_i^{z_i})) = 0$ must hold, implying $x = e_j^{z_j} \wedge (\bigvee_{i:i\neq j} e_i^{z_i})$.
	 \end{proof}
By (2) of the claim, any element $y$ in the sublattice generated by $e_i^\ell$ $(i=1,2,\ldots,k,\ell=0,1,2,\ldots)$ can be written as
\begin{equation}\label{eqn:expression}
y = e_1^{z_1} \vee e_2^{z_2} \vee \cdots \vee e_k^{z_k}
\end{equation}
for $z = (z_1,z_2,\ldots,z_k) \in \ZZ^k_+$.
It suffices to show that the expression (\ref{eqn:expression}) 
is unique.
For $i=1,2,\ldots,k$, let  $z_i' := \max \{ \ell \in \ZZ_+ \mid e^{\ell}_i \preceq y \}$. 
Then $z_i \leq z_i'$ (since $e^{z_i}_i \preceq y$).
Consider $y' := e_1^{z'_1} \vee e_2^{z'_2} \vee \cdots \vee e_k^{z'_k}$. 
Then $y' \preceq y$, which implies $r_x(y') \leq r_x(y)$.
On the other hand, $r_x(y) = z_1+ z_2 + \cdots + z_k \leq z_1' + z_2' + \cdots + z_k' = r_x(y')$. Thus it must hold $z_i = z_i'$ for $i=1,2,\ldots,k$, and $y = y'$. 
\end{proof}

\subsubsection{Parallelism and ends}
Here we introduce a parallel relation for rays, 
and introduce the concept of an end 
as an equivalence class of this relation. 
\begin{Lem}\label{lem:ray}
	Let $x = e^0 \prec e^1 \prec \cdots$ be an $x$-ray. 
	For $y \succeq x$, there is an index $\ell$ 
	such that $e^{\ell} \preceq y$ and $e^{\ell+1} \not \preceq y$.
	In particular, $y=e^{\ell} \vee y \prec  e^{\ell+1} \vee y \prec \cdots$ is a $y$-ray.
\end{Lem}
\begin{proof}
	By (F), 
	there is no infinite chain in any interval.
	Therefore $e^\ell \preceq y$ for all $\ell$ is impossible.
	The latter statement follows from Lemma~\ref{lem:segment}.
\end{proof}

For an $x$-ray 
$(e^\ell) = (x = e^0 \prec e^1 \prec \cdots)$ and $y \succeq x$, 
the $y$-ray in the above lemma is denoted by $(e^\ell) \vee y$.  

An $x$-ray $(e^{\ell})$ and $y$-ray $(f^{\ell})$ are said to be {\em parallel} 
if $(e^{\ell}) \vee (x \vee y) = (f^{\ell}) \vee (x \vee y)$.
We write $(e^{\ell}) \approx (f^{\ell})$ if they are parallel.
Notice that $(e^{\ell})^+ \approx (e^{\ell})$ holds
since $(e^{s})^+ = (e^{s-1})^+ \vee e^{s+1} = 
(e^{s-2})^+ \vee e^{s} \vee e^{s+1} = \cdots = x \vee e^{s+1}$ for $s=0,1,2,\ldots$.
\begin{Lem}\label{lem:equiv}
	The parallel relation $\approx$ is an equivalence relation on the set of all rays.
\end{Lem}
\begin{proof}
	We first show the following claim:
	\begin{Clm}
		Let $(e^{\ell})$ and $(f^\ell)$ be $x$-rays, and let $y \succeq x$.
		Then $(e^{\ell}) = (f^\ell)$ if and only if $(e^\ell) \vee y = (f^\ell) \vee y$.
	\end{Clm}
	\begin{proof}
		The only if part is obvious. We prove the if part.
		Suppose that $(e^{\ell}) \neq (f^\ell)$. We show that $(e^\ell) \vee y \neq (f^\ell) \vee y$.
		We may assume that $y$ covers $x$.
		The above claim is clearly true when $y = e^{1} = f^{1}$.
		Suppose that $y = e^1 \neq f^1$.
		By Proposition~\ref{prop:generated} applied to 
		independent $x$-rays $(e^{\ell})$,$(f^\ell)$, 
		we have $y \vee e^2 = e^2 \neq y \vee f^1$,  
		and $(e^\ell) \vee y \neq (f^\ell) \vee y$.

		Suppose that $y \neq e^1$ and $y \neq f^1$.
		For some $k \geq 0$, 
		we have $e^\ell = f^\ell$ for $\ell \leq k$ and $e^{k+1} \neq f^{k+1}$.
		It suffices to show that $(y \vee e^k)$-rays
		$(y \vee e^k \prec y \vee e^{k+1} \prec \cdots)$
		and $(y \vee f^k \prec y \vee f^{k+1} \prec \cdots)$
		are different.
		So we may consider the case $k = 0$.
		By the above argument, we can assume that $y$, $e^1$, and $f^1$ are different.
		If $y$, $e^1$, and $f^1$ are independent in $[x,x^+]$, 
		then $y \vee e^1$ and $y \vee f^1$ are different, 
		and $(e^\ell) \vee y \neq (f^\ell) \vee y$, as required.
		Suppose that they are dependent; namely
		$y \vee e^1 \vee f^1 = y \vee e^1 = y \vee f^1 = e^1 \vee f^1 =:z$.
		Then $e^2 \neq z$ and $f^2 \neq z$ (since $e^0 \prec e^1 \prec e^2$ is a segment).
		We show that $y \vee e^2$ and $y \vee f^2$ are different.
		By Lemma~\ref{lem:segment}, 
		$e^1 \prec z = e^1 \vee f^1 \prec e^1 \vee f^2 = y \vee f^2$ is a segment.
		If $y \vee e^2 = y \vee f^2$, then $y \vee e^2 = z \vee e^2$ implies that 
		$y \vee f^2$ is the join of $z$ and $e^2$, both covering $e^1$; 
		this contradicts the fact that $e^1 \prec z \prec y \vee f^2$ is a segment.
		Thus $(e^\ell) \vee y \neq (f^\ell) \vee y$.
	\end{proof}

	It suffices to show that $(e^{\ell}) \approx (f^{\ell})$ and $(f^{\ell}) \approx (g^\ell)$ imply 
	$(e^{\ell}) \approx (g^{\ell})$.
	Suppose that $(e^{\ell})$, $(f^{\ell})$, and  $(g^\ell)$
	are $x$-, $y$-, and $z$-rays, respectively.
	Then $(e^{\ell}) \vee (x \vee y) = (f^{\ell}) \vee (x \vee y)$
	and $(f^{\ell})\vee (y \vee z) = (g^\ell) \vee (y \vee z)$.
	This implies that 
	$(e^{\ell}) \vee (x \vee y \vee z) = (f^{\ell}) \vee (x \vee y \vee z) = (g^{\ell}) \vee (x \vee y \vee z)$.
	By the above claim, it must hold $(e^{\ell}) \vee (x \vee z) = (g^{\ell}) \vee (x \vee z)$.
\end{proof}
An equivalence class is called an {\em end}.
Let $E = E^{\cal L}$ denote the set of all ends.
\begin{Lem}
	For an $x$-ray $(e^\ell)$ and $y \in {\cal L}$, 
	there (uniquely) exists a $y$-ray that is parallel to~$(e^{\ell})$.
\end{Lem}
\begin{proof}
	Consider $y' := (y)^{+k} \succeq x$ (Lemma~\ref{lem:preceq}).
	Then $((e^\ell) \vee y')^{-k} \approx (e^\ell) \vee y' \approx (e^\ell)$, 
	implying $((e^\ell) \vee y')^{-k} \approx (e^\ell)$, where $((e^\ell) \vee y')^{-k}$ is a $y$-ray.
\end{proof}

Let $E_x$ denote the set of all $x$-rays.
By the above lemma, for each end $e \in E$, 
there is an $x$-ray $e_x \in E_x$ that is a representative of $e$.
In particular $E_x$ and $E$ are in one-to-one correspondence.
For $e \in E$, the representative of $e$ in $E_x$
is denoted by $e_x = (x = e_x^0 \prec e_x^1 \prec e_x^2 \prec \cdots )$.
In particular,  $E_x = \{e_x \mid e \in E\}$.
\begin{Ex}\label{ex:parallel}
	Consider the parallel relation on rays in $\ZZ^n$ and in $\ZZ^{E,n}$; 
	see Example~\ref{ex:ray}.
	In $\ZZ^n$, 
	two rays $(x+ \ell {\bf 1}_{i}), (y+ \ell {\bf 1}_{j})$ (of form (\ref{eqn:ray_Z^n})) are parallel 
	if and only if $i=j$, i.e., 
	their {\em directions} are the same.
	More generally, 
	two rays $(e^{\ell}),(f^{\ell})$ in $\ZZ^{E,n}$ are parallel if and only if
	$e^{\ell+1} - e^{\ell} = f^{\ell+1} - f^{\ell} = {\bf 1}_{e}$ 
	for some $e \in E$ and large $\ell$. 
	Thus the set $E^{\ZZ^{E, n}}$ of ends is identified with $E$.
\end{Ex}

\subsubsection{Ultrametric on the space of ends}\label{subsub:ultra}

Let $x \in {\cal L}$.
Define $\delta_x: E \times E \to \ZZ_+$ by
\begin{equation*}
\delta_x(e,f) := \sup \{i  \mid e_x^i = f_x^i \} \quad (e,f \in E),  
\end{equation*}
and define $d_x: E \times E \to \RR_+$ by
\begin{equation*}
d_x(e,f) := \exp (- \delta_x(e,f)) \quad (e,f \in E).
\end{equation*}
Observe from Proposition~\ref{prop:generated} that  
two different $x$-rays $(e_x^\ell)$,$(f_x^\ell)$ 
never meet again once they are separated, 
i.e., if $e_x^i \neq f_x^i$ then $e_x^j \neq f_x^j$ for $j > i$.
In particular, all elements in $x$-rays in $E_x$ 
induce a rooted tree with root $x$ in the Hasse diagram of  ${\cal L}$.
From this view, $\delta_x(e,f)$ is the distance between the root $x$ 
and the lowest common ancestor (lca) of $e$ and $f$.
\begin{Prop}\label{prop:ultrametric}
	For $x \in {\cal L}$, we have the following:
	\begin{itemize}
	\item[{\rm (1)}] $d_x$ is an ultrametric on $E$.
	\item[{\rm (2)}] The metric space $(E,d_x)$ is complete.
	\item[{\rm (3)}] For $y \in {\cal L}$, 
	it holds $\alpha^{-1} d_y \leq d_x \leq \alpha d_y$ for a positive constant
    $\alpha :=  \exp ( r[x, x \vee y] + r[y, x \vee y])$.
\end{itemize}
\end{Prop} 
\begin{proof}
(1).  From the view of rooted tree, 
one can easily see that $\delta_x$ satisfies the anti-ultrametric inequality:
\begin{equation*}
\delta_x(e,f) \geq \min (\delta_x(e,g), \delta_x(g,f)) \quad (e,f,g \in E). 
\end{equation*} 
Hence $d_x$ satisfies the ultrametric inequality~(\ref{eqn:ultrametric}).
If $e \neq f$ then $\delta_x(e,f)$ is finite, and $d_x(e,f)$ is nonzero.
This means that $d_x$ is an ultrametric.

(2). 
Consider a Cauchy sequence $(e_i)_{i=1,2,\ldots}$ 
in $E$ relative to $d_x$.
We construct $e \in E$ such that $\lim_{i \rightarrow \infty} d_x(e, e_i) = 0$. 
Let $a^0 := x$. 
For $\ell \in \ZZ_+$, there is 
$n_\ell \in \ZZ_+$ such that $\delta_x(e_i, e_{i'}) \geq \ell$ for $i,i' \geq n_\ell$.
Let $a^{\ell} := f_{x}^{\ell}$ for $f := e_{n_\ell}$.
Then all $(e_i)_x$ for $i \geq n_\ell$ contain~$a^{\ell}$.
Hence $(a^\ell)$ is an $x$-ray such that $(e_{i})$ converges to 
the end $e$ of $x$-ray $(a^\ell)$.

(3).	We first show:
\begin{Clm}
	If $z$ covers $x$, then $\delta_x(e,f) - 1 \leq \delta_z(e,f) \leq \delta_x(e,f) + 1$.
\end{Clm} 
Consider $x$-rays $(e_x^{\ell})$, $(f_x^{\ell})$. 
Suppose that $e_x^k = f_x^k$ and $e_x^{k+1} \neq f_x^{k+1}$, 
i.e., $\delta_x(e,f) = k$.
If $z = e_x^1 = f_x^1$, then $\delta_z(e,f) = \delta_x(e,f) - 1$.
If $z = e_x^1$ and $z \neq f_x^1$, 
then $\delta_z(e,f) = \delta_x(e,f) = 0$ (by Proposition~\ref{prop:generated} and Lemma~\ref{lem:ray}).
So suppose $e_x^1 \neq z \neq f_x^1$.
Then $(e_z^{\ell}) = (e^{\ell}_x)  \vee z = z \prec z \vee e_x^1 \prec \cdots$ and 
 $(f_z^{\ell}) =  (f^{\ell}_x)  \vee z = z \prec z \vee f_x^1 \prec \cdots$.
Also $e_z^{s}$ covers $e_x^s$ and  $f_z^{s}$ covers $f_x^s$. 
Consider $z \vee e_x^k = z \vee f_x^k$.
If $z \vee e_x^k$, $e_x^{k+1}$, and $f_x^{k+1}$ are independent in $[e_x^k, (e_x^k)^+]$, 
then $z \vee e_x^{k+1} \neq z \vee f_x^{k+1}$ and 
$\delta_z(e,f) = \delta_x(e,f) =k$.
If $z \vee e_x^k$, $e_x^{k+1}$, and $f_x^{k+1}$ 
are dependent, i.e.,  
$z \vee e_x^{k+1} = z \vee f_x^{k+1} = z \vee e_x^{k+1} \vee f_x^{k+1}$,  
then $z \vee e_x^{k+2} \neq z \vee f_x^{k+2}$ holds, 
as seen in the proof of  Lemma~\ref{lem:equiv}, 
and $\delta_z(e,f) = \delta_x(e,f) + 1$. 

By the claim, 
we have $\delta_y(e,f) - r[y,x \vee y] - r[x, x \vee y]  \leq 
\delta_x(e,f) \leq \delta_y(e,f) + r[y,x \vee y] + r[x, x \vee y]$.
Then $\alpha^{-1} d_y(e,f) \leq d_x(e,f) \leq \alpha d_y(e,f)$.
\end{proof}
Thus $E^{\cal L}$ is endowed with the topology induced by ultrametric $d_x$, 
which is independent of the choice of $x \in {\cal L}$ by (3).
We will see in Section~\ref{subsec:v->u} that 
$E^{\cal L}$ coincides with 
the Dress-Terhalle completion
when ${\cal L}$ comes from a valuated matroid $(E, \omega)$.

\subsubsection{Realization in $\ZZ^E$}
Here we show that ${\cal L}$ can be realized as a subset of $\ZZ^E$, 
which will be the set of integer points of a tropical linear space.
Let $x \in {\cal L}$.
For $y \succeq x$, the {\em $x$-coordinate} of $y$ 
is an integer vector $y_x \in \ZZ^E_+$ defined by
\begin{equation*}
y_x(e) := \max \{ \ell \in \ZZ_+ \mid e_x^{\ell} \preceq y \} \quad (e \in E).
\end{equation*}

\begin{Lem}\label{lem:y_x}
	For $x \preceq y \preceq z$, we have the following:
	\begin{itemize}
		\item[{\rm (1)}] $z_x = z_y + y_x$.
		\item[{\rm (2)}] $(y)^+_x = y_x + {\bf 1} = y_{(x)^-}$.
		\item[{\rm (3)}] $\displaystyle y = \bigvee_{e \in E} e_x^{y_x(e)}$.
	\end{itemize}
\end{Lem}
\begin{proof}
	(1). 
	It suffices to consider the case where $z$ covers $y$.
	Consider $e \in E$. By semimodularity, $y \vee e_x^{y_x(e)+1}$ covers $y$.
	If $z = y \vee e_x^{y_x(e)+1}$, then $z = e_y^1$ and $z_y(e) = 1$, and 
	$z_x(e) = y_x(e) + 1$, where $z_x(e) > y_x(e) + 1$ is impossible by Lemma~\ref{lem:segment}.
	If $z \neq y \vee e_x^{y_x(e)+1}$, then $z_y(e) = 0$ and $z_x(e) = y_x(e)$ 
	(since $z \vee e_x^{y_x(e)+1} = z \vee (y \vee e_x^{y_x(e)+1})$ covers $z$).

	(2). It is easy to see $(u)^+_u = {\bf 1}$. 
	By (1), we obtain $(y)^+_x = (y)^+_y + y_x 
	= {\bf 1} + y_x = y_x + x_{(x)^-} = y_{(x)^-}$.
	
	(3). Observe from $y \succeq e_x^{y_x(e)}$ that 
	$(\succeq)$ holds; 
	in particular, the right hand side of (3) actually exists.
	We show the equality ($=$).
	Let $u (\preceq y)$ denote the right hand side of~(3).
	Then $u_x = y_x$. From $y_x = y_{u} + u_x$ by (1), we have $y_u = {\bf 0}$.
	Here $y \succ u$ is impossible, otherwise $y_u \neq {\bf 0}$.
\end{proof}

For general $x,y \in {\cal L}$, 
the $x$-coordinate $y_x$ of $y$ is defined by 
\begin{equation*}
y_x := (y)^{+k}_x - k {\bf 1}
\end{equation*}
for an integer $k$ with $y^{+k} \succeq x$.
This is well-defined by Lemma~\ref{lem:y_x}~(2).
Then Lemma~\ref{lem:y_x}~(1) 
and (2) also hold for general $x,y,z$. Indeed, (2) is obvious.
(1) follows from: $(z)^{+k}_x = (z)^{+k}_{(y)^{+\ell}} + (y)^{+\ell}_x$ 
for $x \preceq (y)^{+\ell} \preceq (z)^{+k}$ implies 
$z_x + k {\bf 1} = z_y + k {\bf 1} - \ell {\bf 1} + y_x + \ell {\bf 1}$ 
and $z_x = z_y + y_x$.
By ${\bf 0} = x_x = x_y + y_x$, we have:
\begin{Lem}\label{lem:y_x=-x_y}
	For $x,y \in {\cal L}$, it holds $y_x = - x_y$.
\end{Lem}
%
%
For $x \in {\cal L}$, 
define ${\cal Z}({\cal L},x) \subseteq \ZZ^{E}$ by
\begin{equation}
{\cal Z}({\cal L},x) := \{y_x \mid y \in {\cal L}\}.
\end{equation}
The partial order on ${\cal Z}({\cal L},x)$ is induced by vector order $\leq$ in $\ZZ^{E}$
\begin{Prop}\label{prop:Z(L,x)}
	Let $x \in {\cal L}$. Then
	${\cal L}$ is isomorphic to ${\cal Z}({\cal L},x)$ by $y \mapsto y_x$. 
\end{Prop}
\begin{proof}
	By Lemma~\ref{lem:y_x}~(3),  
	the map $y \mapsto y_x$ is injective on $\{ y \in {\cal L} \mid y \succeq x\}$.
	Via Lemma~\ref{lem:y_x}~(2), 
	it is injective and bijective on ${\cal L}$.
	 
	We show that the order is preserved.
	Suppose that $y \preceq z$. 
	For some $k$, we have $x \preceq y^{+k} \preceq z^{+k}$.
	By Lemma~\ref{lem:y_x}, 
	we have $(z)^{+k}_x = (z)^{+k}_{(y)^{+k}} + (y)^{+k}_x$, and $z_x = z_y + y_x$.
	By $z_y \geq {\bf 0}$, we have $z_x \geq y_x$. 
\end{proof}
Thus ${\cal Z}({\cal L},x)$ is 
a uniform semimodular lattice with vector order $\leq$ and 
ascending operator $x \mapsto x+ {\bf 1}$.

\subsubsection{Matroid at infinity}

Here we introduce matroid structures on the set $E$ of ends.
Suppose that ${\cal L}$ has uniform-rank $n$. 
For $x \in {\cal L}$, 
a subset $I \subseteq E$ of ends is called {\em independent at $x$} 
or {\em $x$-independent}
if $\{e_x^1 \mid e \in I\}$ is independent in $[x,(x)^+]$.
Let ${\cal I}^x = {\cal I}^{{\cal L},x}$ denote the family of all $x$-independent subsets in $E$.
\begin{Lem}\label{lem:M^x}
	$(E, {\cal I}^{x})$ is a loop-free matroid with rank $n$.
\end{Lem}
Indeed, $(E, {\cal I}^{x})$ 
is obtained by adding parallel elements 
to the simple matroid corresponding to 
geometric lattice $[x, (x)^+]$ whose rank is equal to the uniform-rank $n$ of ${\cal L}$.
The matroid ${\bf M}^{x} = {\bf M}^{{\cal L},x} := (E, {\cal I}^{x})$ is called the {\em matroid at $x$}.
Its base family is denoted by ${\cal B}^x$.
Let ${\cal I}^{\infty} := \bigcup_{x \in {\cal L}} {\cal I}^x$ 
be the union of all $x$-independent subsets over all $x \in {\cal L}$.
The goal here is to show the following.
\begin{Prop}\label{prop:matroid_infty}
	$(E, {\cal I}^{\infty})$ is a simple matroid with rank $n$.
\end{Prop}
We call ${\bf M}^\infty := (E, {\cal I}^{\infty})$ 
the {\em matroid at infinity}.
The base family ${\cal B}^{\infty}$ of ${\bf M}^{\infty}$ is given by 
${\cal B}^{\infty} = \bigcup_{x \in {\cal L}} {\cal B}^x$. 
We see in Section~\ref{subsec:u->v} 
that ${\cal B}^{\infty}$ is the domain of the valuated matroid corresponding to ${\cal L}$.
\begin{Ex}
	Consider the case of ${\cal L} = \ZZ^{E,n}$, 
	where $E^{{\cal L}}$ is identified with $E$ (Example~\ref{ex:parallel}).
	Let $x \in \ZZ^{E,n}$. The atoms of 
	$[x, x+{\bf 1}]$ are $x + {\bf 1}_e$ $(e \in E)$ if 
	$|\Argmin x| > m-n +1$, and $x+ {\bf 1}_{\Argmin x}$ and $x+ {\bf 1}_e$ $(e \in E \setminus \Argmin x)$ if $|\Argmin x| = m-n +1$.
	If $|\Argmin x| = m-n +1$, then every subset of atoms is independent.
    Otherwise a subset $J$ of atoms 
	is independent 
	if and only if $|J \cap \Argmin x| \leq n - |E \setminus \Argmin x|$.
	For an end $e \in E$, the atom 
	$e_x^1$ of $[x, x+{\bf 1}]$ is equal to 
	$x+ {\bf 1}_{\Argmin x}$ if 
	$| \Argmin x | = m - n +1$ and $e \in \Argmin x$, 
	and $x + {\bf 1}_e$ otherwise.
	Therefore the matroid ${\bf M}^x = (E, {\cal I}^x)$ at $x$ is given by
	\[
	{\cal I}^x = \{  J \subseteq E  \mid   
	|J \cap \Argmin x| \leq n - |E \setminus \Argmin x| \}.
	\]
	Namely, ${\bf M}^x$ is the direct sum of coloops and the uniform matroid 
	with rank $n - m + |\Argmin x|$. 
	In particular, ${\bf M}^x \subseteq {\bf M}^{\bf 0}$ for every $x \in \ZZ^{E,n}$. 
	Hence the matroid ${\bf M}^{\infty}$ at infinity is equal to ${\bf M}^{\bf 0}$ and is 
	the uniform matroid on the ground set $E$ with rank $n$.
\end{Ex}

We are going to prove Proposition~\ref{prop:matroid_infty}.
\begin{Lem}\label{lem:dependent}
	For $K \subseteq E$ and $x \in {\cal L}$,  we have the following:
	\begin{itemize}
		\item[{\rm (1)}] 
		For any $z \in {\cal L}$ with $z \succeq x$ and 
		$z \not \succeq e_x^1$ $(e \in K)$, if $K \in {\cal I}^z$, then 
		$K \in {\cal I}^x$. 
		 \item[{\rm (2)}] For any $z \in [x,(x)^+]$ with $z \succeq \bigvee_{e \in K} e_x^1$, 
		 it holds $r[z, \bigvee_{e \in K} e_z^1] \geq r[x, \bigvee_{e \in K} e_x^1]$; in particular, 
		 if $K \in {\cal I}^x$, then $K \in {\cal I}^z$.
		 
		 \item[{\rm (3)}] For $I \subseteq K$, let $y := \bigvee_{e \in I} e_x^{1}$. 
		 If $I \in {\cal I}^x$, $K \in {\cal B}^y$, and $e_x^1 \not \preceq y$ for $e \in K \setminus I$, 
		 then $K \in {\cal B}^x$.
	\end{itemize}
\end{Lem}
\begin{proof}
	(1). We show the contrapositive;  
	suppose $|K| > r_x(\bigvee_{e \in K} e_x^1)$, i.e., $K \not \in {\cal I}^x$.
	By $z \not \succeq e_x^1$,  
	it holds $e_z^1 = z \vee e_x^1$ for $e \in K$.
	Then $r_x(z) + |K| >  r_x(z) + r_x(\bigvee_{e \in K} e_x^1) \geq r_x(\bigvee_{e \in K} e_z^1) + r_x(z \wedge \bigvee_{e \in K} e_x^1) =  r_x(z) + r_z(\bigvee_{e \in K} e_z^1) + r_x(z \wedge \bigvee_{e \in K} e_x^1) \geq  r_x(z) + r_z(\bigvee_{e \in K} e_z^1)$.
	Thus $|K| > r_z(\bigvee_{e \in K} e_z^1)$, and $K \not \in {\cal I}^z$.

	(2). Let $y:=  \bigvee_{e \in K} e_x^1$. 
	We can choose an $x$-independent subset $K' \subseteq K$ such that
	$y = \bigvee_{e \in K'} e_x^1$.
	Also we can choose an $x$-independent subset $J \subseteq E \setminus K$ 
	such that $y \vee (\bigvee_{e \in J} e_x^1) = z$.
	Then $K' \cup J$ is $x$-independent. 
	Now $z$ belongs to the sublattice generated 
	by independent $x$-rays $e_x$ $(e \in K' \cup J)$.
	From Proposition~\ref{prop:generated}, we conclude that 
	$K'$ is independent at $z$.
	Hence $r_x[x, \bigvee_{e \in K}e_x^1] = |K'| 
	= r_z[z,\bigvee_{e \in K'}e_z^1] \leq  r_z[z,\bigvee_{e \in K}e_z^1]$.
%

	(3). By $\bigvee_{e \in K}e_x^1 = y \vee \bigvee_{e \in K \setminus I} e_x^1 
	= \bigvee_{e \in K\setminus I} e_y^1$ (since $e_x^1 \not \preceq y$ for $e \in K \setminus I$), 
	we have $r[y, \bigvee_{e \in K} e_x^1] = r[y, \bigvee_{e \in K \setminus I} e_y^1] 
	= |K| - |I| = n - |I|$ (by $K \in {\cal B}^y$).
	Thus $r[x, \bigvee_{e \in K} e_x^1] = r[x,y] + r[y, \bigvee_{e \in K} e_x^1] = n$.
	This implies that $K$ is a base at $x$.  	
\end{proof}

\begin{Lem}\label{lem:independent}
	For $I \subseteq E$ and $x \in {\cal L}$, 
	define $x = x^0,x^1,\ldots$ by
	\begin{equation}\label{eqn:x^k}
	x^{k} := \bigvee_{e \in I} e_{x}^k \quad (k=0,1,2,\ldots).
	\end{equation}
	If $I \in {\cal I}^{\infty}$, 
	then there is $m \geq 0$ such that $I$ is independent at $x^m$.
\end{Lem}
\begin{proof}
	By the definition of ${\cal I}^{\infty}$, there is $y \in {\cal L}$ 
	such that $I$ is independent at $y$.
	We can assume that $y \succeq x$ (Lemma~\ref{lem:preceq}).
	Consider the $x$-coordinate $y_x \in \ZZ^E$ of $y$, and 
	let $z := \bigvee_{e \in I} e_x^{y_x(e)} (\preceq y)$.
	By Lemma~\ref{lem:y_x}~(1), 
	it holds $y_z(e) = 0$ for all $e \in I$.
	This means that $e^1_z \not \preceq y$ for all $e \in I$.
	Therefore, by Lemma~\ref{lem:dependent}~(1) and $I \in {\cal I}_y$, 
	$I$ is independent at $z$.
	By $z \not \succeq e_x^{y_x(e) + 1}$ and Lemma~\ref{lem:ray}, 
    it holds $e_z^{l} = z \vee e_x^{y_x(e) + l}$ for $e \in I$ and $l \geq 0$. 
	
	Let $m := \max_{e \in I} y_x(e)$.
	Then $x^m = \bigvee_{e \in I} e_x^{y_x(e)} \vee e_x^{m}
	=  \bigvee_{e \in I} z \vee e_x^{m} = \bigvee_{e \in I} e_{z}^{m - y_z(e)}$.
	Thus $x^m$ belongs to the sublattice generated 
	by independent $z$-rays, which implies that $I$ is independent at $x^m$.	
\end{proof}

\begin{Lem}\label{lem:x^k'}
	For $I \subseteq E$ and $x \in {\cal L}$, define $x = x^0,x^1,\ldots$ by $(\ref{eqn:x^k})$.
	Then we have 
	\begin{equation}\label{eqn:x^k'}
	x^k = \bigvee_{e \in I} e_{x^{k-1}}^1 \quad (k=1,2,\ldots).
	\end{equation}
\end{Lem}
\begin{proof}
	We show by induction on $k$ that $e_x^{k} \not \preceq x^{k-1}$ for $e \in I$.
	This implies $e_{x^{k-1}}^{1} = x^{k-1} \vee e_x^{k}$ 
	by Lemma~\ref{lem:ray}, and implies
	(\ref{eqn:x^k'}):  
	$x^k := \bigvee_{e \in I} e_x^k 
	= \bigvee_{e \in I} e_x^{k-1} \vee e_x^k 
	= \bigvee_{e \in I} x^{k-1} \vee e_x^k
	= \bigvee_{e \in I} e_{x^{k-1}}^1$. 
	
	For $e \in I$, by induction, $e_x^{k-1} \not \preceq x^{k-2}$.
	Then $e_x^{k} \vee x^{k-2} = e_{x^{k-2}}^{2}$ (by Lemma~\ref{lem:ray}).
	If $e^k_x \preceq x^{k-1}$, then 
	$e_{x^{k-2}}^2 = e_x^k \vee x^{k-2} \preceq x^{k-1}$, and
	$x^{k-2} = e_{x^{k-2}}^0 \prec e_{x^{k-2}}^1 \prec e_{x^{k-2}}^2 
	\preceq x^{k-1} = \bigvee_{e \in I} e_{x^{k-2}}^1 \preceq (x^{k-2})^+$, contradicting $e_{x^{k-2}}^2 \not \in [x^{k-2}, (x^{k-2})^+]$. 
	Thus $e_x^{k} \not \preceq x^{k-1}$, as required.	
\end{proof}

\begin{proof}[Proof of Proposition~\ref{prop:matroid_infty}]
We verify the axiom of independent sets.
Choose $I,J \in {\cal I}^{\infty}$ with $|I| < |J|$.
By the definition of ${\cal I}^{\infty}$,
there is $x \in {\cal L}$ with $I \in {\cal I}^x$.
Consider $x^1 := \bigvee_{e \in I} e_x^1$ and $y^1 := \bigvee_{e \in J} e_x^1$.
If $y^1 \not \preceq x^1$, 
then we can choose $e^* \in J \setminus I$ 
with $(e^*)^1_{x} \not \preceq x^1$, and $I + e^*$ is independent at $x$; 
$I + e^* \in {\cal I}_x \subseteq {\cal I}_{\infty}$, as required.

So suppose $y^1 \preceq x^1$. For $k=1,2,\ldots$, 
let $x^k := \bigvee_{e \in I} e_x^k$, and let $y^k :=  \bigvee_{e \in J} e_x^k$.
By Lemma~\ref{lem:dependent}~(2) and Lemma~\ref{lem:x^k'}, 
$I$ is independent at all $x^k$.
By Lemma~\ref{lem:independent} and $J \in {\cal I}^{\infty}$, 
there is $\ell$ such that $J$ is independent at all $y^k$ for $k \geq \ell$. 
With Lemma~\ref{lem:x^k'}, it holds
$r[x^{k},x^{k+1}] = r[x^k, \bigvee_{e \in I} e_{x^k}^1]= |I| < |J| 
= r[y^k, \bigvee_{e \in J} e_{y^k}^1] = r[y^k,y^{k+1}]$ for $k \geq \ell$. 
For large $k$, the increase of the height of $y^k$ is greater than that of $x^k$.
Therefore there is $k^*$ such that $y^{k^*} \preceq x^{k^*}$ 
and $y^{k^*+1} \not \preceq x^{k^*+1}$.
This implies that $x^{k^*+1} \not \succeq x^{k^*} \vee y^{k^*+1} 
= x^{k^*} \vee \bigvee_{e \in J} e_{y^{k^*}}^{1} \preceq \bigvee_{e \in J} e_{x^{k^*}}^1$. 
Thus $\bigvee_{e \in J} e_{x^{k^*}}^1 \not \preceq \bigvee_{e \in I}e_{x^{k^*}}^1 (= x^{k^*+1})$, and there is $e^* \in J \setminus I$ with $I + e^* \in {\cal I}_{x^{k^*}}$, as above.

For distinct $e,f \in E$ and $x \in {\cal L}$, 
let $y := e_x^{\delta_x(e,f)} = f_x^{\delta_x(e,f)}$. 
Then $e_{y}^1 \neq f_{y}^1$; see Section~\ref{subsub:ultra},
This means that $\{e,f\}$ is independent on ${\bf M}^{\infty}$.  
Thus ${\bf M}^{\infty}$ is a simple matroid.
\end{proof}

\begin{Lem}\label{lem:key}
	Let $x \in {\cal L}$. 
	For a bounded vector $c \in \ZZ^E_+$, let $y := \bigvee_{e \in E} e_x^{c(e)}$.
	Then there is $B \in {\cal B}^y$ such that $y_x(e) = c(e)$ for $e \in B$, and
	\begin{equation}
	y = \bigvee_{e \in B} e_x^{c(e)}.
	\end{equation}
\end{Lem}
Notice that $\bigvee_{e \in E} e_x^{c(e)}$ 
exists by $\bigvee_{e \in E} e_x^{c(e)} \preceq (x)^{+\max_{e \in E}c(e)}$.
\begin{proof}
	We use the induction on $\max_{e \in E} c(e)$. 
	Define $c' \in \ZZ_{+}^{E}$ by $c'(e) := \max \{ c(e) - 1,0\}$.
    Let $z := \bigvee_{e \in E} e_x^{c'(e)}$.
    Let $Z := \{ e \in E \mid e_x^{c(e)} \not \preceq z \}$.
    Then, for $e \in Z$, it holds $c(e) > 0$ 
    and $e_x^{c(e)-1} \preceq z \not \preceq e_x^{c(e)}$.
    This implies 
    \begin{equation*}
    z_x(e) = c(e) - 1 \quad (e \in Z).
    \end{equation*}
    Now $y = \bigvee_{e \in E} z \vee e_{x}^{c(e)} = \bigvee_{e \in Z} e_z^1$.
    Consider the matroid ${\bf M}^z$ at $z$. Then 
    \begin{equation}\label{eqn:y_z(e)}
    y_z(e) = 
    \left\{ \begin{array}{ll}
    1 & {\rm if}\ e \in {\rm cl} (Z) (\Leftrightarrow e_z^1 \preceq y), \\
    0 & {\rm otherwise},
    \end{array}\right. \quad (e \in E).
    \end{equation}
    Therefore, 
    by Lemma~\ref{lem:y_x}~(1) we have 
    \begin{equation}\label{eqn:y_x(e)}
    y_x(e) = y_z(e) + z_x(e) = c(e) \quad (e \in Z).
    \end{equation}

    By induction, there is $B' \in {\cal B}^z$ 
    such that $z_x(e) = c'(e)$ for $e \in B'$ and $z = \bigvee_{e \in B'} e_x^{c'(e)}$.
    Let $I := \{ e \in B' \mid  c(e) > 0 \}$.
    By $c(e) - 1 = c'(e)= z_x(e)$ for $e \in I$, 
    it holds $I \subseteq Z$.    
    By $I \in {\cal I}^z$ (from $B' \in {\cal B}^z$), 
    there is $J \in {\cal I}^z$ such that $I \subseteq J \subseteq Z$ 
    and $y = \bigvee_{e \in J} e_z^1$ (i.e., ${\rm cl}(J) = {\rm cl} (Z)$).
    By $z =  \bigvee_{e \in I} e_x^{c(e)-1}$ and $I \subseteq J \subseteq Z$, 
    it holds $y = \bigvee_{e \in J} e_z^1 = \bigvee_{e \in J} z \vee e_x^{c(e)} 
    = \bigvee_{e \in I} e_x^{c(e)-1} \vee \bigvee_{e \in J} e_x^{c(e)} = \bigvee_{e \in J} e_x^{c(e)}$.
    Therefore, if $J \in {\cal B}^z$, 
    then $J \in {\cal B}^y$ (by Proposition~\ref{prop:generated}), 
    and by (\ref{eqn:y_x(e)}) $J$ is a desired subset.
    Suppose not. 
    By the independence axiom for $B',J \in {\cal I}^z$ with $|B'| > |J|$ 
    we can choose a subset $K \subseteq B' \setminus J$ 
    with $J \cup K \in {\cal B}^z$. 
    Necessarily $K$ is disjoint with ${\rm cl}(Z)$.
    Then $B := J \cup K$ is a desired base in ${\cal B}^y$.
    Indeed, 
    by $K \subseteq B' \setminus I$, we have $0 = c(e) = c'(e) = z_x(e)$ for $e \in K$.
    By $K \cap {\rm cl}(Z) = \emptyset$ and (\ref{eqn:y_z(e)}), we have $y_z(e) = 0$ for $e \in K$. 
    Thus $y_x(e) = y_z(e) + z_x(e) = 0 = c(e)$ for $e \in K$; then $y_z(e) = c(e)$ for $e \in B = J \cup K$.
    Also $y = \bigvee_{e \in J} e_x^{c(e)} = \bigvee_{e \in B} e_x^{c(e)}$. 
\end{proof}

\subsubsection{$\ZZ^n$-skeletons}
Let $x \in {\cal L}$, and $B \in {\cal B}^x$.
By Proposition~\ref{prop:generated}, 
the
sublattice ${\cal S}^x(B)$ generated by elements in $x$-rays $e_x \in B$ is 
isomorphic to  $\ZZ_+^n$, where $n$ is the uniform rank of ${\cal L}$.
This sublattice is closed under the ascending operation. 
Define sublattice ${\cal S}(B)$ by
\begin{equation*}
{\cal S}(B) := \bigcup_{k \in \ZZ} ({\cal S}^x(B))^k.
\end{equation*}
Then ${\cal S}(B)$ is isomorphic to $\ZZ^n$ 
with $(y)^+ = y + {\bf 1}$ for $y \in {\cal S}(B)$ (identified with $\ZZ^n$).
We call ${\cal S}(B)$ the {\em $\ZZ^n$-skeleton} generated by $B$.
The next lemma shows that ${\cal S}(B)$ 
is independent of the choice of $x$, 
and is well-defined for $B \in {\cal B}^{\infty}$.
\begin{Lem}\label{lem:S(B)}
	For $B \in {\cal B}^x$, it holds
	${\cal S}(B) = \{ y \in {\cal L} \mid B \in {\cal B}^y  \}$.
\end{Lem}
\begin{proof}
	From Proposition~\ref{prop:generated},
	the inclusion $(\subseteq)$ is obvious. We show the converse. 
	Let $y \in {\cal L}$ with $B \in {\cal B}^y$.
	We may assume that $y \succeq x$ 
	by considering $(y)^{+k}$ and by $({\cal S}(B))^{+k} = {\cal S}(B)$.
	Let $y' := \bigvee_{e \in B} e_x^{y_x(e)}$. 
	Then $y' \preceq y$.
	We show $y' = y$. 
	Suppose not:  $y' \prec y$. 
	There is an atom $a$ of $[y',(y')^+]$ with $a \preceq y$; necessarily $a \neq e_{y'}^1$ for $e \in B$.
	By Lemma~\ref{lem:dependent}~(1), $B$ is also a maximal independent set at $y'$.
	Hence $r_{y'}(a \vee \bigvee_{e \in B}  e_{y'}^1 ) = r_{y'}((y')^+)= n$  
	and $n - 1 = r_{a} (\bigvee_{e \in B} (a \vee e_{y'}^1)) = r_a(\bigvee_{e \in B} e_{a}^1)$.
	Namely $B$ is dependent at $a$ with $a \preceq y \not \succeq e_a^1$ for $e \in B$.
	By Lemma~\ref{lem:dependent}~(1), 
	$B$ is dependent at $y$, contradicting $B \in {\cal B}^y$.
\end{proof}

\subsection{Valuated matroids from uniform semimodular lattices}\label{subsec:u->v}
Let ${\cal L}$ be a uniform semimodular lattice with uniform-rank $n$.
For $x \in {\cal L}$ and $B \in {\cal B}^{\infty}$, 
define $x_B \in {\cal L}$ as the maximum element $y \in {\cal S}(B)$ with $y \preceq x$:
\begin{equation*}
x_B := \bigvee \{y \in {\cal S}(B) \mid y \preceq x\}.
\end{equation*}
The maximum element $x_B$ indeed exists by (F) and the fact that 
${\cal S}(B)$ is a sublattice. 
Now define $\omega = \omega^{{\cal L},x}: {\cal B}^{\infty} \to \ZZ$ by 
\begin{equation}\label{eqn:omega}
\omega(B) := - r[x_B,x] \quad (B \in  {\cal B}^{\infty}).
\end{equation}
This quantity $\omega(B)$ is the negative of a ``distance" 
between $x$ and ${\cal S}(B)$. 
%
%
One of the main theorems is as follows: 

\begin{Thm}\label{thm:main1}
Let ${\cal L}$ be a uniform semimodular lattice with uniform-rank $n$, 
and let $x \in {\cal L}$.
Then $\omega = \omega^{{\cal L}, x}$ is 
a complete valuated matroid with rank $n$, 
where 
\begin{itemize}
	\item[{\rm (1)}] ${\cal T}(\omega) \cap \ZZ^E$ is isomorphic to ${\cal L}$, and
	\item[{\rm (2)}] ${\cal T}(\omega)$ is a geometric realization of simplicial complex ${\cal C}({\cal L})$ consisting of all chains 
	$x^0 \prec x^1 \prec \cdots \prec x^m$ with $x^m \preceq (x^0)^+$.
\end{itemize}
\end{Thm}
\begin{Ex}
	We consider the case of ${\cal L} = \ZZ^{E,n}$.
	For $B \in {\cal B}^{\infty}$ (an arbitrary $n$-element subset of $E$), 
	a point $y \in \ZZ^{E,n}$ belongs to ${\cal S}(B)$
	if and only if $E \setminus \Argmin y \subseteq B$.
	Then the $\ZZ^n$-skeleton ${\cal S}(B)$ is actually isomorphic to 
	$\ZZ^B \simeq \ZZ^n$; 
	indeed consider the map $\ZZ^{B} \ni x \mapsto \bar x \in {\cal S}(B)$, 
	where $\bar x(e) := x(e)$ for $e \in B$ and 
	$\bar x(e) := \min_{e \in B} x(e)$ for $e \in E \setminus B$.
	Let $x \in \ZZ^{E,n}$. Then one can observe that
	$x_B$ is given by 
	\begin{equation*}
	x_B(e) := \left\{  \begin{array}{ll}
	\min_{f \in E} x(f) & {\rm if}\ e \in (E \setminus \Argmin x) \setminus B, \\
	x(e) & {\rm otherwise}.
	\end{array} \right.
	\end{equation*} 
	Observe from the covering relation in $\ZZ^{E,n}$ (Example~\ref{ex:Z^E,n}) that $r[x_B,x]$ is equal to $\sum_{e \in (E \setminus \Argmin x) \setminus B} (x(e) - \min_{f \in E} x(f))$.  Observe further that this quantity is also written 
	as $\max_{B \in {\cal B}^{\infty}} x(B') - x(B)$.
	Thus $\omega(B) =  x (B) - \max_{B \in {\cal B}^{\infty}} x(B')$, and
    $\omega$ is projectively equivalent to 
	the trivial valuation on the uniform matroid.
\end{Ex}

To prove Theorem~\ref{thm:main1}, we show several properties of $x_B$.
\begin{Lem}\label{lem:x_B1} 
	Let $x,y \in {\cal L}$ with $y \preceq x$, and $B \in {\cal B}^{\infty}$.
\begin{itemize}
	\item[{\rm (1)}]  $y=x_B$ if and only if $B \in {\cal B}^y$ and $x_y(e) = 0$ for all $e \in B$.
	\item[{\rm (2)}] $x_B \preceq y$ if and only if
	$x_y(e) = 0$ for all $e \in B$.
\end{itemize}	
\end{Lem}
\begin{proof}
	(1). Suppose that $y \in {\cal S}(B)$ ($\Leftrightarrow$ $B \in {\cal B}^y$).
	Then $y \preceq x_B$ and
	$x_B = \bigvee_{e \in B} e_y^{(x_B)_y(e)}$ (by Proposition~\ref{prop:generated}).
	By Lemma~\ref{lem:y_x}~(1), it holds $x_y = x_{x_B} + (x_B)_y$.
	Therefore, if $x_y(e) = 0$ for $e \in B$, 
	then $(x_B)_y(e) = 0$ for $e \in B$ and 
	$x_B = \bigvee_{e \in B} e_y^{(x_B)_y(e)} = y$.
	If $x_y(e) > 0$ for some $e \in B$, then 
	then $\bigvee_{e \in B} e_y^{x_y(e)}$ belongs to ${\cal S}(B)$, is greater than $y$, 
	and is not greater than $x$, i.e., $y \neq x_B$. 	
	In particular, $x_{x_B}(e) = 0$ for $e \in B$.
	
	(2). The only-if part follows from $x_{x_B} = x_{y} + y_{x_B}$ and 
	$x_{x_B}(e) = 0$ of all $e \in B$.
	We show the if part.
	Suppose that $B$ is dependent at $y$ (otherwise $y = x_B$ by (1)).
	Define the sequence $y = y^0, y^1, y^2,\ldots$ by
	\begin{equation}\label{eqn:y^k}
	y^k := (\bigvee_{e \in B} e_y^{k})^{-k} \quad (k=0,1,2,\ldots).
	\end{equation}
	Then it holds that
	\begin{equation}\label{eqn:y^k'}
	y^k = (\bigvee_{e \in B} e^1_{y^{k-1}})^{-1} \quad (k=1,2,\ldots).
	\end{equation}
	Indeed, let $z^{k} := \bigvee_{e \in B} e_y^{k}$. 
	Then $y^k = (z^{k})^{-k} = (\bigvee e_{z^{k-1}}^1)^{-k} 
	= (\bigvee (e_{z^{k-1}}^1)^{-k+1})^{-1} 
	= (\bigvee e_{(z^{k-1})^{-k+1}}^1)^{-1} = (\bigvee e_{y^{k-1}}^1)^{-1}$, 
	where the second equality follows from Lemma~\ref{lem:x^k'} 
	and the forth one follows from 
	the observation $(e_u^1)^{-1} = e_{u^{-1}}^1$.
	Since $\bigvee_{e \in B} e^1_{y^{k-1}} \in [y^{k-1},(y^{k-1})^+]$, 
	we have $y^k \preceq y^{k-1}$.
	In particular, $x \succeq y \succeq y^1 \succeq y^2 \succeq \cdots$ holds.
	By Lemma~\ref{lem:y_x}~(1) and (2), 
	it holds $x_{y^k}(e) = x_{y^{k-1}}(e) + (y^{k-1})_{y^k}(e) = 
	x_{y^{k-1}}(e) + (1 - 1) = x_{y^{k-1}}(e)$ for $e \in B$. 
	This implies $x_{y^k}(e) = x_y(e) = 0$ for all $e \in B$.  
	By Lemma~\ref{lem:independent} and (\ref{eqn:y^k}), there is $\ell$ 
	such that $B \in {\cal B}^{y^\ell}$. 
	By (1), we have $y^\ell = x_B$, and $x_B \preceq y$, as required.
\end{proof}

\begin{Lem}\label{lem:x_B2}
	For $x,y \in {\cal L}$ with $y \preceq x$, we have the following:
	\begin{equation*}
	r(x_B) + \sum_{e \in B} y_x(e) 
	\left\{
	\begin{array}{ll}
	= r(y)  & {\rm if}\  y \in {\cal S}(B) ( \Leftrightarrow B \in {\cal B}^y), \\
	< r(y)  & {\rm otherwise}.
	\end{array}
	\right. \quad (B \in {\cal B}^{\infty}).
	\end{equation*}
\end{Lem}
\begin{proof}
	Suppose that $y \in {\cal S}(B)$.
	By Lemmas~\ref{lem:y_x}~(1) and~\ref{lem:x_B1}~(1), 
	 $\bigvee_{e \in B} e_y^{x_y(e)}$ is equal to $x_B$.
	By Proposition~\ref{prop:generated}, 
	 $r[y,x_B] = \sum_{e \in B} x_y(e)$.
	Therefore $r(y) + \sum_{e \in B} x_y(e) = r(x_B)$ holds, 
	which implies $r(y) = r(x_B) + \sum_{e \in B} y_x(e)$ 
	by $y_x = - x_y$; see Lemma~\ref{lem:y_x=-x_y}.

	Suppose that $y \not \in {\cal S}(B)$.
	Let $y' := \bigvee_{e \in B} e_y^{x_y(e)}$.
	Then $x_y = x_{y'} + y'_y$ and $y'_y(e) = x_y(e)$ for $e \in B$ imply
	$y'_x(e) = 0$ for $e \in B$.
	By Lemma~\ref{lem:x_B1}~(2), 
	we have $x_{B} \preceq y' \preceq x$, and 
	\begin{eqnarray*}
	&& r[y,y'] \leq \sum_{e \in B} x_y(e), \\
	&& r(x_B) \leq r(y'). 
	\end{eqnarray*}
	It suffices to show that one of the inequalities is strict.
	If $y' \succ x_B$, then $(<)$ holds in the second inequality.
	Suppose that $y' = x_B$, and suppose to the contrary that  
	equality holds in the first inequality.
	Let $I:= \{e \in B \mid x_y(e) > 0 \} (\neq \emptyset)$, and 
	let $y'' := \bigvee_{e \in I} e_y^{x_y(e) - 1}$.
	Then $y' = x_B = \bigvee_{e \in I}e_{y''}^{1}$.
	By the equality in the first inequality and Lemma~\ref{lem:y_x}(1), 
	$I$ must be independent at $y''$, 
	and $e_{y''}^1 \not \preceq y'$ for $e \in B \setminus I$ 
	(otherwise $x_y(e) > 0$ for $e \in B \setminus I$ ). 
	By Lemma~\ref{lem:dependent}~(3), $B$ is independent at~$y''$.
	Also $r[y,y''] = \sum_{e \in B} \max \{ x_y(e) - 1,0\}$ holds. 
	By repeating this argument (to $y''$), we eventually obtain a contradiction that 
	$B$ is independent at~$y \not \in {\cal S}(B)$.
\end{proof}

\begin{proof}[Proof of Theorem~\ref{thm:main1}]
	Observe that $\omega$ is upper-bounded.
	By Lemma~\ref{lem:murota}, 
	we show that for any bounded vector $c \in \ZZ^E$ 
	the maximizer family ${\cal B}_{\omega + c}$
	is a matroid base family.
	
	Suppose that $c = y_x$ for some $y \preceq x$.
	By Lemma~\ref{lem:x_B2}, 
	the maximizer family ${\cal B}_{\omega + c}$ 
	is nothing but ${\cal B}^{y}$.
	
	Suppose that $c$ is general. 
	From ${\cal B}_{\omega + c}  = {\cal B}_{\omega + c + k {\bf 1}}$, 
	we can assume that $c \geq 0$. 
	Let $y := \bigvee_{e \in E} e_x^{c(e)}$.
	By Lemma~\ref{lem:key}, there is $B \in {\cal B}^y$ 
	such that $y = \bigvee_{e \in B} e_x^{c(e)}$ and $c(e) = y_x(e)$ for $e \in B$.
	Let $\tilde c := y_x$.
	Then $\tilde c \geq c$. Thus 
	$- r[x_{B'},x] + \sum_{e \in B'} c(e) \leq - r[x_{B'},x] + \sum_{e \in {B'}} \tilde c(e)$
	for arbitrary $B' \in {\cal B}^{\infty}$, and 
	the equality holds for $B$ by $c(e) = y_x(e) = \tilde c(e)$ $(e \in B)$.
	Since $B \in {\cal B}^y = {\cal B}_{\omega + \tilde c}$ (by above), 
	the maximum of $\omega + c$ 
	is the same as that of $\omega + \tilde c$. 
	This implies that ${\cal B}_{\omega +c} \subseteq {\cal B}_{\omega + \tilde c}$.
	Now ${\cal B}_{\omega +c}$ is viewed as 
	the maximizer family of a linear function 
	$B \mapsto \sum_{e \in B} (c - \tilde c)(e)$
	over the matroid base family ${\cal B}_{\omega + \tilde c}$, and 
	is a matroid base family, as required.
	
	(1) follows from Proposition~\ref{prop:Z(L,x)} and the next claim.
	\begin{Clm}
		${\cal T}(\omega) \cap \ZZ^E = {\cal Z}({\cal L}, x)$.
	\end{Clm} 
	\begin{proof}
	For $c = y_x \in {\cal Z}({\cal L}, x)$, the maximizer family 
	${\cal B}_{\omega + c}$ is equal to ${\cal B}^y$, as seen above.
	The matroid ${\bf M}^y = (E, {\cal B}^y)$ is loop-free (Lemma~\ref{lem:M^x}).
	Hence $(\supseteq)$.
	
	Let $c \in \ZZ^E_{+}$ with $c \not \in {\cal Z}({\cal L}, x)$.
	Consider $\tilde c$ as above. Then $\tilde c \geq c$, and $\tilde c \neq c$.
	As seen above, 
	$\max_{B} -r[x_{B},x] + \sum_{e \in B} c(e) = \max_{B} - r[x_{B},x] + \sum_{e \in B} \tilde c(e)$. 
	This means that an element $e \in E$ with $\tilde c(e) > c(e)$
	cannot belong to any maximizer in ${\cal B}_{\omega +c}$.
	Namely $e$ is a loop in ${\cal B}_{\omega + c}$. 
	Thus $c \not \in {\cal T}(\omega) \cap \ZZ^n$, 
	implying  $(\subseteq)$.	
	\end{proof}
	(2) is a corollary of this claim and Lemma~\ref{lem:chain_of_flats}~(2).
	By Proposition~\ref{prop:matroid_infty}, $(E,\omega)$ is a simple valuated matroid.
	Lemma~\ref{lem:d=D} in the next section shows that 
	topologies on $E$ induced by $d_x$ and by $D_p$ from $\omega$ coincide.
	By Proposition~\ref{prop:ultrametric}~(2), 
	$\omega$ is complete.
\end{proof}

\subsection{Uniform semimodular lattices from valuated matroids}\label{subsec:v->u}

The main statement for the uniform 
semimodular lattice of a valuated matroid 
is as follows.
\begin{Thm}\label{thm:main2}
Let $(E,\omega)$ be
an integer-valued valuated matroid with rank $n$.
Then ${\cal L}(\omega):= {\cal T}(\omega) \cap \ZZ^E$ is a uniform semimodular lattice with 
uniform-rank $n$, in which the following hold:
\begin{itemize}
	\item[{\rm (1)}] The ascending operator is equal to $x \mapsto x + {\bf 1}$.
	\item[{\rm (2)}] A height function $r$ is given by 
	\[
	x \mapsto \max_{B \in {\cal B}} (\omega + x)(B).
	\]
	\item[{\rm (3)}] The meet $\wedge$ and the join $\vee$ are given by
	\begin{eqnarray*}
	x \wedge y  &= & \min (x,y), \\
	x \vee y  &= & \bigwedge \{ z \in {\cal L}(\omega) \mid x \leq z \geq y \} \quad (x,y \in {\cal L}(\omega)).
	\end{eqnarray*}
	\item[{\rm (4)}] For $x \in {\cal L}(\omega)$, the valuated matroid 
	$(E^{{\cal L}(\omega)}, \omega^{{\cal L}(\omega),x})$
	is a completion of a valuated matroid projectively equivalent to $(E, \omega)$.
\end{itemize}
\end{Thm}

The rest of this section is to devoted to the proof.
Let ${\bf M} = (E, {\cal B})$ be the underlying matroid of $\omega$.
By (TC$_{+{\bf 1}}$), 
if $x\in {\cal L}(\omega)$ then $x+{\bf 1} \in {\cal L}(\omega)$. 
We first show that the interval $[x,x+{\bf 1}]$ in ${\cal L}(\omega)$ 
is a geometric lattice 
corresponding to ${\bf M}_{\omega + x}$.

\begin{Lem}\label{lem:[x,x+1]} Let $x \in {\cal L}(\omega)$.
	\begin{itemize}
		\item[{\rm (1)}] $[x,x+{\bf 1}]$ 
		is isomorphic to the lattice of flats of ${\bf M}_{\omega + x}$, 
		where the isomorphism is given by the map $x + {\bf 1}_{F} \mapsto F$.
		\item[{\rm (2)}] $y \in {\cal L}(\omega)$ covers $x$ if and only if $y = x + {\bf 1}_F$
		for a parallel class $F$ in ${\bf M}_{\omega + x}$.
	\end{itemize}
\end{Lem}

\begin{proof}
	(1). By replacing $\omega$ by $\omega+x$,
	we can assume $x = {\bf 0}$.
	By Lemma~\ref{lem:integer-valued},
	for a flat $F$ of ${\cal B}_{\omega}$, and any $e \in F$ and $f \not \in F$ 
	we can choose $B \in  {\cal B}_{\omega} \cap {\cal B}_{\omega + {\bf 1}_{F}}$ 
	containing $e,f$. This implies $x + {\bf 1}_{F} \in {\cal L}(\omega)$.
	Suppose that $F$ is not a flat of  ${\cal B}_{\omega}$.
	Consider $e \in {\rm cl}(F) \setminus F$.
	Then $\max \{ |B \cap (F + e)| \mid B \in {\cal B}_{\omega}\} 
	= \max \{ |B \cap F| \mid B \in {\cal B}_{\omega}\}$.
	This implies that $\max_B (\omega + {\bf 1}_{F})(B) 
	= \max_B (\omega + {\bf 1}_{F + e})(B)$.
	Thus no base in ${\cal B}_{\omega+ {\bf 1}_F}$ contains $e$, implying $x + {\bf 1}_{F} \not \in {\cal L}(\omega)$.

	(2). By (1), it suffices to the only-if part.
	We first show that for $F \subseteq E$ and $e \in E \setminus F$,
	if $e$ is a loop in ${\bf M}_{\omega}$ then so is ${\bf M}_{\omega + {\bf 1}_F}$.
	Choose $B \in {\cal B}_{\omega}$ with maximal $B \cap F$. 
    By Lemma~\ref{lem:integer-valued} 
    it holds $B \in {\cal B}_{\omega+ {\bf 1}_F}$.
	Suppose (to the contrary) that there is a base in ${\cal B}_{\omega + {\bf 1}_F}$
	containing $e$. 
	By the exchange axiom there is $f \in B$ such that
	$B + e - f \in {\cal B}_{\omega + {\bf 1}_F}$.
	Then $B + e - f \not \in {\cal B}_{\omega}$, and 
	$\omega(B+e-f) \leq \omega(B) - 1$.
	By $e \not \in F$, 
	it holds $|(B + e- f) \cap F| \leq |B \cap F|$.
	Therefore $(\omega+{\bf 1}_{F})(B+e-f) < (\omega+{\bf 1}_{F})(B)$, 
	contradicting $B + e - f \in {\cal B}_{\omega + {\bf 1}_F}$.
	Thus no base in ${\cal B}_{\omega + {\bf 1}_F}$ contains $e$.
	
	Let $y = x + \sum_{i} {\bf 1}_{F_i}$ for $F_1 \supseteq F_2 \supseteq \cdots \supseteq F_m$.
	By repeated uses of the above property, one can see that
	$F_1$ must be a flat in ${\cal B}_{\omega + x}$;
	otherwise $e \in {\rm cl}(F_1) \setminus F_1$ 
	is a loop in ${\cal B}_{\omega + y}$.
	Consider the parallel class $F$ of $e \in F_1$ in ${\bf M}_{\omega+x}$.
	By (1), 
	$x+ {\bf 1}_F$ belongs to ${\cal L}(\omega)$.
	Therefore $x \leq x+ {\bf 1}_F \leq y$, implying $y = x+ {\bf 1}_F$.	 
\end{proof}

\begin{proof}[Proof of Theorem~\ref{thm:main2} (1-3)]
	First we show (2) that
	a height function $r$ of ${\cal L}(\omega)$ is given by  $x \mapsto \max_{B \in {\cal B}} (\omega + x)(B)$.
	Consider $x,y \in {\cal L}(\omega)$ such that $y$ covers $x$.
	By Lemma~\ref{lem:[x,x+1]}~(2),  $y= x+ {\bf 1}_F$ 
	for a parallel class $F$.  
	Then ${\cal B}_{\omega+y} \supseteq \{ B \in {\cal B}_{\omega + x} \mid |B \cap F| = 1\} (\neq \emptyset)$ by Lemma~\ref{lem:integer-valued}.  
	Therefore $r(y) = r(x)+1$.
	
	Next we show that ${\cal L}(\omega)$ is a lattice with property (3).
	Let $x,y \in {\cal L}(\omega)$, and let $z := \min (x,y)$.
	By the tropical convexity (Lemma~\ref{lem:trop_convexity}),
	$z$ belongs to ${\cal L}(\omega)$, and necessarily 
	$x \wedge y = z$.
	 By Lemma~\ref{lem:[x,x+1]}~(2) and (2) shown above,
	 $x - z$ and $y - z$  are upper-bounded.
	 This implies that $\max(x,y) - x$ and $\max(x,y) - y$ are upper-bounded.
	 Thus $\{z \in {\cal L}(\omega) \mid z \geq \max(x,y)\}$ is nonempty; 
	 for example, consider $x+ \alpha {\bf 1}$ for large $\alpha$.
     By this fact and the existence of a height function,  
	 	$\bigwedge \{z \in {\cal L}(\omega) \mid z \geq \max(x,y)\}$ exists, and is the join of $x,y$.

	 	By Lemma~\ref{lem:[x,x+1]}, 
	 	if $a,b$ cover $a \wedge b$, then $a \vee b$ covers $a,b$.
	 	Hence ${\cal L}(\omega)$ is semimodular (Lemma~\ref{lem:semimodular}).
	 	The property (1) is also an immediate corollary of the same lemma.
	 	The map $x \mapsto x + {\bf 1}$ is obviously an automorphism. 
	 	Thus ${\cal L}(\omega)$ 
	 	is a uniform semimodular lattice. 
	 	The uniform-rank is equal to the rank of 
	 	$[x, x+{\bf 1}]$ that is equal to the rank of~${\bf M}$.
\end{proof}

To show the property (4), 
we have to study the relationship between $E$ 
and the space $E^{{\cal L}(\omega)}$ of ends 
in ${\cal L}(\omega)$.
\begin{Lem}\label{lem:normal_ray}
	Let $(a^\ell)$ be a ray in ${\cal L}(\omega)$.
	\begin{itemize}
	\item[{\rm (1)}] There is a decreasing sequence $F_0 \supseteq F_1 \supseteq \cdots$
	of nonempty subsets in $E$ such that 
	\begin{equation*}
	a^{\ell+1} = a^{\ell} + {\bf 1}_{F_{\ell}} \quad (\ell = 0,1,\ldots),
	\end{equation*}
	where $F_{\ell}$ is a parallel class of ${\bf M}_{\omega + a^\ell}$.
	\item[{\rm (2)}] If $\bigcap_{\ell} F_{\ell}$ is nonempty, 
	then $\bigcap_{\ell} F_{\ell}$ is a parallel class of ${\bf M}$.
	\end{itemize}
\end{Lem}
\begin{proof}
	(1). By Lemma~\ref{lem:[x,x+1]} (2), 
	$F_\ell$ is a parallel class of ${\bf M}_{\omega + a^\ell}$.
	It suffices to show $F_0 \supseteq F_1$.
	Here $F_0 \cap F_1 = \emptyset$ is impossible, since
	otherwise $a^2 \in [a^0, a^0+{\bf 1}]$ contradicting the fact that $(a^\ell)$ is a ray.
	Suppose $F_1 \setminus F_0 \neq \emptyset$.
	Choose $e \in F_0 \cap F_1$ and $f \in F_1 \setminus F_0$.
	Then there is a base $B \in {\cal B}_{\omega + a^0}$ containing $e,f$.
	By Lemma~\ref{lem:integer-valued}, $B$ is also a base in $B_{\omega + a^1}$.
	Namely $e,f$ are independent in ${\bf M}_{\omega + a^1}$.
	However this is a contradiction to the fact that $F_1$ is a parallel class 
	of  ${\bf M}_{\omega + a^1}$.
	
	(2). 
	Suppose that there are distinct non-parallel elements 
	$e,f \in \bigcap_{\ell} F_{\ell}$.
	There is $B \in {\cal B}$ containing $e,f$.
	Then $(\omega + a^{\ell+1})(B) 
	- (\omega + a^{\ell})(B) \geq 2$.
	On the other hand, $r(a^{\ell+1}) - r(a^{\ell}) = 1$ for all $\ell$.
	Recall Theorem~\ref{thm:main2}~(2) that a hight function $r$ is given by $x \mapsto \max_{B} (w+x)(B)$.  
	Then $B$ must be in ${\cal B}_{\omega + a^{\ell'}}$ for some $\ell'$; 
	this is a contradiction to the fact that 
	$e,f$ are parallel in ${\bf M}_{\omega + a^{\ell}}$ for all $\ell$. 
\end{proof}

In the case of (2), ray $(a^\ell)$ is said to be {\em normal} 
and {\em have $\infty$-direction $F = \bigcap_{\ell} F_{\ell}$}.
\begin{Lem}\label{lem:normal_ray2}
	\begin{itemize}
		\item[{\rm (1)}] Two normal rays are parallel if and only if they have the same $\infty$-direction.
		\item[{\rm (2)}] For $x \in {\cal L}(\omega)$ and a parallel class $F$ of ${\bf M}$, 
		there is a normal $x$-ray having $\infty$-direction $F$.
	\end{itemize}
\end{Lem}
\begin{proof}
	(1). Let $(a^\ell)$ be a normal ray having $\infty$-direction $F$, 
	and let $y \in {\cal L}(\omega)$ with $a^{\ell'+1} \not \succeq y \succeq a^{\ell'}$.
	We show that ray $(a^\ell) \vee y = 
	(y = a^{\ell'} \vee y \prec a^{\ell'+1} \vee y \prec \cdots)$ 
	is a normal ray having $\infty$-direction $F$.
	By Lemma~\ref{lem:normal_ray}~(1),  
	we can suppose that $y \vee a^{\ell'+k+1} = y \vee a^{\ell' + k} + {\bf 1}_{G_k}$ for $G_k \subseteq E$.
	By $\min (y,a^{\ell'+1}) = y \wedge a^{\ell'+1} = a^{\ell'}$ and 
	$a^{\ell'+1} = a^{\ell'} + {\bf 1}_{F_{\ell'}}$, 
	it holds $y \vee a^{\ell'+1} - y \geq \max (y, a^{\ell'+1}) - y = {\bf 1}_{F_{\ell'}}$.
	Necessarily $G_0 \supseteq F_{\ell' + 1}$.
	Consequently $G_k \supseteq F_{\ell'+k+1}$ for all $k$. 
	Therefore $\bigcap_k G_k$ contains $F$, and must be equal to $F$, 
	since  $\bigcap_k G_k$ is also a parallel class of ${\bf M}$ (Lemma~\ref{lem:normal_ray}~(2)). 
	Thus $(a^\ell) \vee y$ has $\infty$-direction~$F$.
	The only-if part is immediate from this property.
	The if-part also follows from this property and the observation that 
	if two normal rays at the same starting point
	have the same $\infty$-direction, then the two rays must be equal.
	
	(2). Note that 
	$F$ is a rank-$1$ subset in ${\bf M}_{\omega + y}$ for every $y \in {\cal L}(\omega)$. 
	Let $a^0 := x$. For $\ell=0,1,2,\ldots$, 
	define $F^{\ell}$ as ${\rm cl}\, (F)$ in ${\bf M}_{\omega + a^{\ell}}$, 
	and $a^{\ell+1} := a^\ell + {\bf 1}_{F^{\ell}}$.
	Then $(a^{\ell})$ is a ray, 
	since $a^{\ell + 1} \geq a^{\ell-1} + 2 {\bf 1}_F$ and $a^{\ell + 1} \not \in [a^{\ell-1}, a^{\ell-1}+{\bf 1}]$.
	Also $(a^{\ell})$ is normal with $\infty$-direction $F$ 
	(since parallel class $\bigcap_{\ell} F^{\ell}$ contains $F$ and equals $F$). 
\end{proof}
In the case where $\omega$ is simple, 
by associating $e \in E$ with the end having $\infty$-direction $\{e\}$, 
we can regard $E$ as a subset of $E^{{\cal L}(\omega)}$.
Then each local matroid ${\bf M}_{\omega + x}$ 
is the restriction of ${\bf M}^{{\cal L},x}$ to $E$:
\begin{Lem}\label{lem:restriction} 
	For $x \in {\cal L}(\omega)$, it holds
	${\cal B}_{\omega+x} = \{ B \in {\cal B}^{{\cal L},x} \mid B \subseteq E \}$.
\end{Lem}
\begin{proof}
	By Lemma~\ref{lem:[x,x+1]},
	$B \in {\cal B}_{\omega + x}$ if and only if $x+ {\bf 1}_{F_e}$ $(e \in B)$
	are independent atoms in geometric lattice $[x,x+{\bf 1}]$, 
	where $F_e$ is the parallel class of $e$ in ${\bf M}_{\omega + x}$.
	If $e \in E$ is regarded as a normal ray, then $e_x^1 = x + {\bf 1}_{F_e}$.
	From this, we see the equality to hold.
\end{proof}
We verify that $d_x$ and $D_p$ induce the same topology on the set $E$ of normal rays.
\begin{Lem}\label{lem:d=D}
	Suppose that $\omega$ is simple.
	For $x \in {\cal L}(\omega)$, if $r(x) = 0$, 
	then $-x \in {\cal TS}(\omega)$, and $D_{-x}(e,f) = d_x(e,f)$ for $e,f \in E$.	
\end{Lem}
\begin{proof}
	The fact $-x \in {\cal TS}(\omega)$ follows from (\ref{eqn:TS}) 
	and $r(x) = \max_{B} (\omega + x)(B)$.
	It suffices to show that 
	for two normal rays $e,f \in E$, it holds
	\begin{equation}\label{eqn:h(x)-max}
	\delta_x(e,f) = - \max \{ (\omega+x)(B) \mid  B \in {\cal B}: \{e,f\} \subseteq B \}\  (\geq 0).
	\end{equation}
	Consider 
	the sequence $x = x^0,x^1,\ldots$ defined by 
	$x^{i+1} := e_{x^i}^1 \vee f_{x^i}^1 = e_{x}^{i+1} \vee f_x^{i+1}$; recall Lemmas~\ref{lem:independent} and \ref{lem:x^k'}.
	Then $\delta_x(e,f)$ is the minimum index $i^*$ 
	such that $r(x^{i^*+1}) = r(x^{i^*}) + 2$ 
	or equivalently that there is 
	$B \in {\cal B}_{\omega + x^{i^*}}$ with $e,f \in B$, i.e., $r(x^{i^*}) = (\omega+x^{i^*})(B)$.
	For $i \leq i^*$, it holds
	$r(x^i) = r(x^{i-1}) + 1$, and $(\omega+x^i)(B) = (\omega+x^{i-1})(B) + 2$
	for base $B \in {\cal B}$ with $e,f \in B$.
	Therefore the index $i^*$ must be the right hand side of (\ref{eqn:h(x)-max}).
\end{proof}
\begin{Lem}\label{lem:dense}
	The set $E$ of normal rays is dense in $E^{{\cal L}(\omega)}$.
\end{Lem}
\begin{proof}
	Consider a ray $e \in E^{{\cal L}(\omega)}$. Let $x \in {\cal L}$.
	Then $x$-ray
	$e_x$ is represented as in Lemma~\ref{lem:normal_ray} for 
	some decreasing sequence $F_1 \supseteq F_2 \supseteq \cdots$
	of nonempty subsets in $E$.
	For each $i$, choose $e_i \in F_i$.
	Then the sequence $(e_i)$ of normal rays 
	satisfies $\lim_{i \rightarrow \infty} d_x(e,e_i) = 0$.
\end{proof}

\begin{proof}[Proof of Theorem~\ref{thm:main2}(4)]
We can assume that $\omega$ is simple.
Let $x \in {\cal L}(\omega)$.
By Lemmas~\ref{lem:d=D} and \ref{lem:dense}, 
$E^{{\cal L}(\omega)}$ coincides with the Dress-Terhalle completion of $E$.	
Finally we verify the linear equivalence between $\omega$ and $\omega^{{\cal L},x}$ (restricted to $E$).
\begin{Clm}
	For $B \in {\cal B}$, it holds 
	$(\omega+x)(B) = r(x_B) = \omega^{{\cal L},x}(B) + r(x)$. 
\end{Clm}
\begin{proof}
	It suffices to show the first equality; the second follows from the definition~(\ref{eqn:omega}) 
	of $\omega^{{\cal L},x}$.
	Consider the sequence $x = x^0 \succeq x^1 \succeq \cdots$ defined by
	$x^{i} := (\bigvee_{e \in B} e^1_{x^{i-1}})^{-1} = \bigvee_{e \in B} e^1_{x^{i-1}} - {\bf 1}$.
	As seen in the proof of Lemma~\ref{lem:x_B1} 
	(see~(\ref{eqn:y^k}) and (\ref{eqn:y^k'})), for some $k$ it holds $x^k = x_B$.
	We prove the statement by induction on $k$.
	In the case of $k = 0$, $x = x_B$, 
	$B \in {\cal B}^{{\cal L},x}$, and
	$B \in {\cal B}_{\omega +x}$ by Lemma~\ref{lem:restriction}. 
	Then $r(x_B) = r(x) = (\omega+x)(B)$, 
	implying the base case.
	
	Suppose $k > 0$. Notice $(x^1)_B = x_B$.
	By induction, $(\omega+x^1)(B) = r(x_B)$.
	By definition of $x^k$, it holds $x^1(e) = x(e)$ for $e \in B$.
	Therefore, $(\omega+x)(B) = (\omega + x^1)(B) + (x - x^1)(B) = r(x_B)$, as required.  
\end{proof}
Note the constant term $r(x)$ is represented as 
linear term $(r(x)/n){\bf 1}$.
Thus $\omega$ is projectively equivalent to the restriction of $\omega^{{\cal L}(\omega),x}$ to $E$. 
This completes the proof of Theorem~\ref{thm:main2}.
\end{proof}

\section{Examples}\label{sec:example}

\paragraph{Tree metrics.} 
{\em Tree metrics} may be viewed as valuated matroids of rank $2$; see e.g.,~\cite{DressTerhalle98}. 
We here study 
tree metrics from our framework of uniform semimodular lattice.
Let $T = (V,E)$ be a tree, and let $X$ be a subset of vertices of $T$.
Let ${\cal B} := \{ \{u,v\} \subseteq X \mid u \neq v \}$. 
Then ${\bf M} = (X, {\cal B})$ is a uniform matroid of rank $2$. 
Define $d: {\cal B} \to \ZZ$ by
\begin{equation*}
d(u,v) := \mbox{the number of edges in the unique path in $T$ connecting $u$ and $v$},
\end{equation*}
where $d(\{u,v\})$ is written as $d(u,v)$. 
Then the classical four-point condition of tree-metrics says 
\begin{equation*}
d(u,v) + d (u',v') \leq 
\max \{ d(u,v') + d (u',v), d(u',v) + d(u,v') \}
\end{equation*}
for distinct $u,v,u',v' \in X$.
This is nothing but the exchange axiom (EXC).
Thus $d$ is a valuated matroid on ${\bf M}$.
 
Let us construct the corresponding uniform semimodular lattice in a combinatorial way.
First delete all redundant vertices not belonging to 
the (shortest) path between any pair of $X$.
Fix a vertex $z \in V$ (as a root).
Next, for each $u \in X$, 
consider an infinite path $P_u$ (with $V(P_u) \cap V(T) = \emptyset$) having a vertex $u'$ of degree one.
Glue $T$ and $P_u$ by identifying $u$ and $u'$.
Let ${\cal L}$ denote the union of 
$V \times 2 \ZZ$ and $E \times (2 \ZZ+1)$.
For each $(uv,k) \in E \times (2 \ZZ+1)$, 
consider binary relations (directed edges) 
$(uv,k) \leftarrow (u,k+1)$, $(uv,k) \leftarrow (v,k+1)$, 
$(u,k-1) \leftarrow (uv,k)$, and $(v,k-1) \leftarrow (uv,k)$.
The partial order $\preceq$ on ${\cal L}$ is induced by 
the transitive closure of $\leftarrow$.
Then ${\cal L}$ is a uniform (semi)modular lattice of uniform-rank $2$, where  
the ascending operator is given by
$(x,k) \mapsto (x,k+2)$; see~\cite[Example~3.2]{HH18a}.
See also Figure~\ref{tree} for this construction.
	\begin{figure}[t]
		\begin{center}
			\includegraphics[scale=0.6]{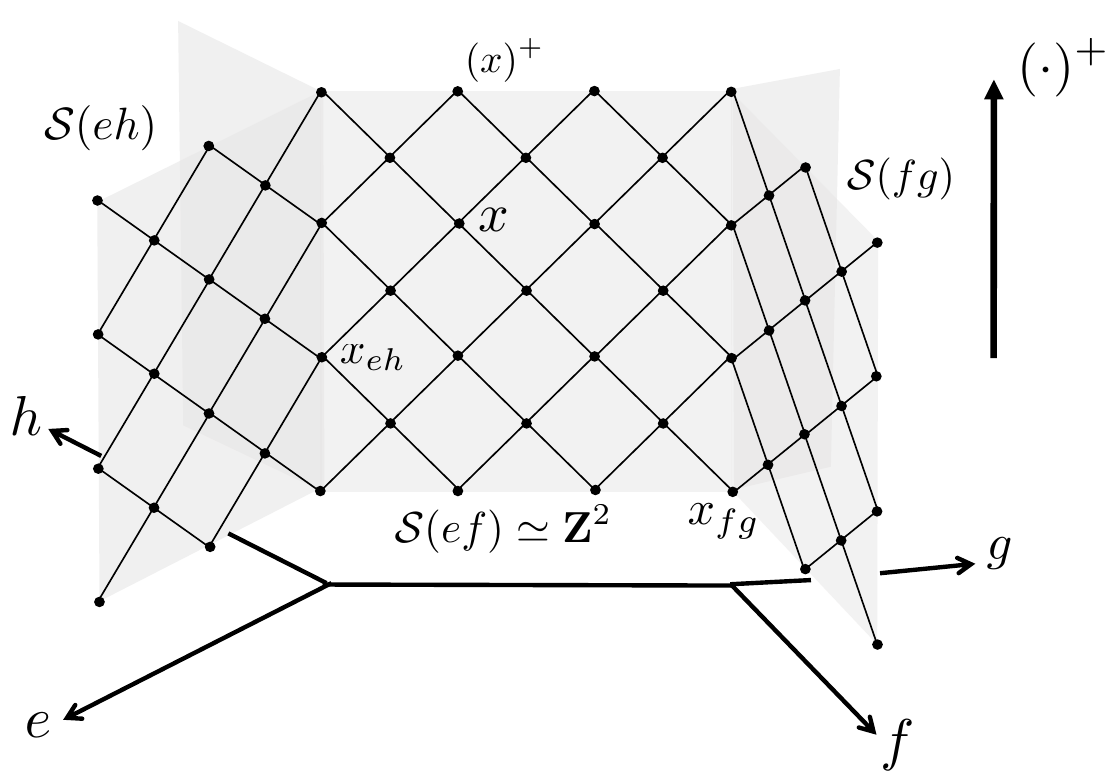}
			\caption{The uniform semimodular lattice for a tree}
			\label{tree}
		\end{center}
	\end{figure}\noindent

Ends are naturally identified with $P_u$ $(u \in X)$. 
In particular ${\cal B} = {\cal B}_{\infty}$.
For two ends $P_u, P_v$, there is a simple path $P$ of $T$ 
containing $P_u,P_v$.
The $\ZZ^2$-skeleton ${\cal S}(\{u,v\})$
is the sublattice of ${\cal L}$ induced by the union of $V(P) \times \ZZ$ and $E(P) \times (2\ZZ+1)$. In the figure, base $\{u,v\}$ is abbreviated as $uv$.  
Let $x := (z,0)$.
For the lowest common ancestor $z_{u,v}$ of $u, v$ in $T$, 
$x_{\{u,v\}}$ is given by $(z_{u,v}, - 2 d(z,z_{u,v}))$. 
Thus the valuated matroid $\omega = \omega^{{\cal L},x}$ is given by 
\begin{equation*}
\omega (u,v) = - 2 d(z, z_{u,v}) \quad (\{u,v\} \in {\cal B}).
\end{equation*}
From the relation $- 2 d(z, z_{u,v}) = d(u,v) - d(z,u) - d(z,v)$,
we see the projective-equivalence between $\omega$ and $d$.

\paragraph{The Bergman fan of a matroid.}
A matroid ${\bf M} = (E,{\cal B})$ is 
naturally viewed as a $\{0, -\infty\}$-valued valuated matroid $\omega$ by
\begin{equation*}
\omega(B) := 0 \Leftrightarrow B \in {\cal B}.
\end{equation*}
In this case, the tropical linear space ${\cal T}(\omega)$ 
is a polyhedral fan in $\RR^E$, 
which is called the {\em Bergman fan} of ${\bf M}$~\cite{ArdilaKlivans06}.
Suppose that ${\bf M}$ is simple.
Ardila and Klivans~\cite{ArdilaKlivans06} showed 
that ${\cal T}(\omega)$ admits a simplicial cone decomposition 
isomorphic to the order complex of the lattice of flats of ${\bf M}$.
Indeed, ${\cal T}(\omega)$ is explicitly written as:
\begin{equation}\label{eqn:Bergman}
{\cal T}(\omega) = \RR{\bf 1} + \bigcup_{F_0,F_1,\ldots, F_k} \mbox{the conical hull of } 
	{\bf 1}_{F_1}, {\bf 1}_{F_2}, \ldots {\bf 1}_{F_k}, 
\end{equation}
where the union is taken over all chains 
$\emptyset \neq F_1 \subset F_2 \subset \cdots \subset F_k \neq E$ $(k \geq 1)$
of (nontrivial) flats.
Notice that every $x \in {\cal T}(\omega)$
has a unique expression 
$x = \mu {\bf 1} + \sum_{i} \lambda_i {\bf 1}_{F_i}$ 
for $\mu \in \RR$, a chain of flats 
$\emptyset \neq F_1 \subset F_2 \subset \cdots \subset F_k \neq E$, 
and positive coefficients $\lambda_1,\lambda_2,\ldots,\lambda_k$.
Therefore ${\cal T}(\omega)$ is viewed as 
a conical geometric realization of the order complex of 
the geometric lattice of flats.
This actually holds for our infinite setting. 
Indeed, a point $x$ in ${\cal T}(\omega)$
is precisely a linear objective vector 
for which the maximizer family over bases of ${\bf M}$
has no loop. (In \cite{ArdilaKlivans06},  
the Bergman fan is defined by the minimizer family and hence
is the negative of (\ref{eqn:Bergman}).)
From this, one can verify by the same argument 
of the proof of Lemma~\ref{lem:chain_of_flats}~(2)
that if $x \in {\cal T}(\omega)$ and $\alpha \in \RR$
then $F = \{ e \in E \mid x(e) \geq \alpha\}$ is a flat of ${\bf M}$, 
which implies (\ref{eqn:Bergman}).

The family ${\cal L}(\omega)$ of integer points of ${\cal T}(\omega)$, 
the uniform semimodular lattice of~$\omega$, is given by
\begin{equation}\label{eqn:BergmanZ}
{\cal L}(\omega) = \ZZ{\bf 1} + \bigcup_{F_0,F_1,\ldots, F_k} \mbox{the integer conical hull of } 
{\bf 1}_{F_1}, {\bf 1}_{F_2}, \ldots {\bf 1}_{F_k}, 
\end{equation}
where the union is taken over chains of flats, as above, and the integer conical hull means the set of all nonnegative integer combinations.
The matroid at the origin ${\bf 0} \in {\cal L}(\omega)$ is equal to ${\bf M}$, 
and the matroid at $x = \mu {\bf 1} + \sum_{i=1}^k \lambda_i {\bf 1}_{F_i} \in {\cal L}(\omega)$ 
is a submatroid of ${\bf M}$ that is the direct product of $({\bf M}|F_{i+1})/F_{i}$ 
for $i=0,1,\ldots,k$ (with $F_0 = \emptyset$ and $F_{k+1} = E$), where  
$|$ and $/$ mean the restriction and contraction, respectively. See \cite{ArdilaKlivans06}.  
The ${\bf 0}$-rays for 
are given by $(k{\bf 1}_{e})_{k \in \ZZ_+}$ $(e \in E)$.
So $E$ is naturally identified with the space $E^{{\cal L}(\omega)}$ of ends.
In particular, $\omega$ is a complete valuated matroid.
The matroid at infinity is also equal to~${\bf M}$.

This construction of the Bergman fan 
gives rise to a general construction of a uniform semimodular lattice 
from a geometric lattice.
Indeed, the right hand side of (\ref{eqn:BergmanZ}) is definable 
for an arbitrary geometric lattice ${\cal L}$.
In this way, every geometric lattice is 
extended canonically to a uniform semimodular lattice.

\paragraph{Representable valuated matroids.} 

Let $K$ be a field, and $K(t)$ the field of rational functions with an indeterminate $t$.
The {\em degree} $\deg (p/q)$ of $p/q \in K(t)$ 
with polynomials $p,q$ is defined by $\deg (p) - \deg (q)$.
Consider the vector space $K(t)^n$ over $K(t)$.
Let $E$ be a subset of $K(t)^n$, and let ${\cal B}$ be 
the family of $K(t)$-bases $B \subseteq E$ of $K(t)^n$. 
Then ${\bf M} = (E, {\cal B})$ is a matroid.
Define $\omega = \omega^E:{\cal B} \to \ZZ$ by
\begin{equation*}
\omega^E(B) := \deg \det (B) \quad (B \in {\cal B}),
\end{equation*}
where $B \in {\cal B}$ is regarded as 
a nonsingular $n \times n$ matrix consisting of vectors in $B$. 
Then $\omega^E$ is a valuated matroid. 
Such a valuated matroid is called {\em representable} (over $K(t)$).
In fact, this construction of valuated matroids is possible 
even if $K$ is a skew field; see~\cite{HH18c}.

A tropical interpretation \cite{MurotaTamura01,Speyer08,SpeyerSturmfels04} of 
${\cal L}(\omega) = {\cal T}(\omega) \cap \ZZ^E$  
is the set of degree vectors $(\deg (q^{\top} e): e \in E)$ for all $q \in K(t)^n$, 
where we need to add $-\infty$ to $\ZZ$ for $\deg (0) := - \infty$.
We here consider a different algebraic interpretation, 
which is essentially the same as the concept of 
the {\em membrane} due to Keel and Tevelev~\cite{KeelTevelev06}
and
is viewed as an analogue of: 
The lattice of flats of the matroid represented by a matrix $M$
is the lattice of vector spaces spanned by  columns of $M$.

Let $K(t)^-$ denote the ring of elements $p/q$ in $K(t)$ with $\deg (p/q) \leq 0$.
Then $K(t)^n$ is also viewed as a $K(t)^-$-module. 
For a subset $F \subseteq K(t)^n$, 
let $\langle F \rangle$ denote the $K(t)^-$-module generated by $F$, 
i.e., $\langle F \rangle = \{ \sum_{u \in F'} \lambda_u u \mid \lambda_u \in K(t)^-, F' \subseteq F: |F'| < \infty \}$. 
Also, for $z \in \ZZ^F$, let $F^z := \{t^{z(u)} u \mid u \in  F\}$. 

Suppose that $E \subseteq K(t)^n$ contains a $K(t)$-basis of $K(t)^n$.
Let ${\cal B} \subseteq 2^{E}$ be the family of $K(t)$-bases, 
which is the underlying matroid of $(E,\omega)$.
Define the family ${\cal L}(E)$ of $K(t)^-$-submodules of $K(t)^n$ by 
\begin{equation*}
{\cal L}(E) := \{ \langle E^z \rangle \mid z \in \ZZ^E \}.
\end{equation*}
The {\em membrane} of $E$~\cite{KeelTevelev06} in 
is the projection of ${\cal L}(E)$
by the equivalence relation $\simeq$ defined by 
$L \simeq L' \Leftrightarrow L = t^k L'$ $(\exists t \in \ZZ)$; see also \cite{JoswigSturmfelsYu07,Zhang2018}.
The partial order on ${\cal L}(E)$ is defined as the inclusion relation.
For $L \in {\cal L}(E)$, define $z^L \in \ZZ^E$ by
\begin{equation*}
	z^L(p) :=\max \{\alpha \in \ZZ \mid t^\alpha p \in L \} \quad (p \in E).
\end{equation*}
\begin{Prop}\label{prop:L(E)}
${\cal L}(E)$ is a uniform semimodular lattice that is isomorphic to ${\cal L}(\omega^E)$
by the maps $L \mapsto z^L$ and $z \mapsto \langle E^z \rangle$, 
where the following hold:
\begin{itemize}
	\item[{\rm (1)}] The ascending operator is given by $L \mapsto t L$.
	\item[{\rm (2)}]
	The $\ZZ^n$-skeleton ${\cal S}(B)$ of $B \in {\cal B}$ is 
	equal to ${\cal L}(B) (:= \{\langle B^z \rangle \mid z \in \ZZ^B \})$.
	\item[{\rm (3)}] A height function~$r$ of ${\cal L}(E)$ is given by
	\begin{equation*}
	r(L) = \deg \det (Q) \quad (L \in {\cal L}(E)),  
	\end{equation*}
	where $Q$ is a $K(t)^-$-basis of $L$.
	\item[{\rm (4)}] For $x \in {\cal L}(\omega)$, 
    it holds
	\begin{eqnarray*}
	\langle E^x \rangle_B & = & \langle B^x \rangle, \\
	\omega^{{\cal L}(\omega),x}(B) & = & (\omega^E + x)(B) - r(\langle E^x \rangle) 
	\quad (B \in {\cal B}).
	\end{eqnarray*}
\end{itemize}
\end{Prop}
Note that a part of the claim, e.g., 
the equivalence between ${\cal L}(E)$ and ${\cal L}(\omega^E)$,  
follows from results in \cite{KeelTevelev06}. 
Here we prove Proposition~\ref{prop:L(E)} in a self-contained way. 
For $F \subseteq E$ and $x \in \ZZ^E$, 
we denote $F^{x|_F}$ by $F^x$.
The proof uses the following basic lemma.
\begin{Lem}\label{lem:basis}
	$\langle E \rangle$ is a free $K^-(t)$-module having any 
	$B \in {\cal B}_{\omega}$ as a basis.  
\end{Lem}

\begin{proof}
Choose any 	$B \in {\cal B}_{\omega}$.
Since $B$ is a $K(t)$-basis of ${K(t)}^n$, 
every element $u \in E$ is represented as 
$u = B \lambda$ for $\lambda \in K(t)^n$, where $B$ is regarded as a matrix.
By Cramer's rule, the $i$-th component $\lambda_i$ of $\lambda$
 is equal to $\det (B^{i}) / \det (B)$, 
where $B^i$ is obtained from $B$ by replacing the $i$-th column with $u$.
Then $\deg (\lambda_i) 
= \deg \det (B^{i})  - \deg \det (B) = \omega(B^i) - \omega(B) \leq 0$
by $B \in {\cal B}_{\omega}$. 
This means that $\lambda \in K^-(t)$. 
Consequently $\langle E \rangle$ is a free $K^-(t)$-module of basis $B$.
\end{proof}

\begin{proof}[Proof of Proposition~\ref{prop:L(E)}]
Obviously we have $L = \langle E^{z^L} \rangle$.
We show that $z^L \in {\cal L}(\omega)$.
Suppose indirectly that $p \in E$ is a loop in ${\cal B}_{\omega + z^L}$.
By the above lemma, for any 
$B \in {\cal B}_{\omega+ z^L}$, $B^{z^L}$ is a basis of $L$. 
Consider equation $B^{z^L} \lambda = t^{z^L(p)} p$. 
By using Cramer's rule as above, we have 
$\deg \lambda_i = (\omega+z^L)(B - e_i + p) - (\omega+z^L)(B) \leq - 1$ for each $i=1,2,\ldots,n$, 
where $e_i$ is the $i$-th column of $B$.
The inequality follows from the fact
that $p$ is a loop in ${\cal B}_{\omega + z^L}$.
This means that $t^{z^L(p)+1} p$ also belongs to $\langle B^{z^L} \rangle = L$. 
This is a contradiction to the definition of $z^L$.
Thus $z^L \in {\cal L}(\omega)$.
Also $L \mapsto z^L$ is the inverse of $z \mapsto \langle E^z \rangle$.
Indeed, $z^{E^z} \geq z$. If $z^{E^z}(e) > z(e)$, 
then one can see as above that $e$ does not belong to any base of ${\cal B}_{\omega + z}$, 
contradicting $z \in {\cal L}(\omega)$.  

(1). This follows from $z^{tL} = z^L + {\bf 1}$.

(2). Observe that the sublattice 
${\cal L}(B) = \{\langle B^z \rangle \mid z \in \ZZ^B\}$ 
of ${\cal L}(E)$ is isomorphic to~$\ZZ^n$. 
By Lemma~\ref{lem:basis}, 
we have $B \in {\cal B}_{\omega + z^L}$ for $L = \langle B^z \rangle \in {\cal L}(B)$.
By Lemma~\ref{lem:S(B)}, we have $B \in {\cal S}(B)$. 
Thus ${\cal L}(B) \subseteq {\cal S}(B)$.
Both ${\cal L}(B)$ and ${\cal S}(B)$ are isomorphic to $\ZZ^n$ 
with the same ascending operator.
Consequently, it must hold ${\cal L}(B) = {\cal S}(B)$.

(3). Suppose that $L'$ covers $L$.
We can choose $B \in {\cal B}$ with $L,L' \in {\cal S}(B)$.
Necessarily $L' = \langle B^{z'} \rangle$ and $L = \langle B^z \rangle$ 
for $z - z' = {\bf 1}_{e}$ for some $e \in B$.
Then $\deg \det B^{z'}= \deg \det B^z + 1$.

(4). It obviously holds that $\langle B^x \rangle \in {\cal L}(B) = {\cal S}(B)$, and 
$\langle B^x \rangle \subseteq \langle E^x \rangle_B$. 
Suppose indirectly that the inclusion is strict. 
Then, for some $e \in B$, it holds $\langle B^{x+{\bf 1}_e} \rangle \subseteq \langle E^x \rangle_B \subseteq \langle E^x \rangle$.
This means that $\langle E^{x+{\bf 1}_e} \rangle = \langle E^{x} \rangle$.
However this is a contradiction to $x = z^{\langle E^{x} \rangle}$.

From the definition, we have
$\omega^{{\cal L}(\omega),x}(B) 
= - r[\langle B^x \rangle, \langle E^{x} \rangle]
= \deg \det (B^x)  - r(\langle E^{x} \rangle) 
= \omega^{E^x}(B) - r(\langle E^{x} \rangle) = (\omega^{E}+x)(B) - r(\langle E^{x} \rangle)$ 
\end{proof}

\paragraph{Modular valuated matroids and Euclidean buildings.}
Analogous to a modular matroid --- 
a matroid whose lattice of flats is a modular lattice, 
a {\em modular valuated matroid} is defined as 
an integer-valued valuated matroid  $(E,\omega)$
such that the corresponding ${\cal L}(\omega)$ is a uniform modular lattice.
The companion work~\cite{HH18a} showed that uniform modular lattices
and Euclidean buildings of type A are cryptomorphically equivalent in the following sense. 
For a uniform modular lattice ${\cal L}$, 
define equivalence relation $\simeq$ on ${\cal L}$ by $x \simeq y$ if $x = (y)^{+k}$ for some $k$. Then the simplicial complex ${\cal C}({\cal L})$ modulo $\simeq$
is a Euclidean building of type~A; 
recall Theorem~\ref{thm:main1} for the simplicial complex ${\cal C}({\cal L})$ of short chains of ${\cal L}$.
Conversely, every Euclidean building of type A is obtained in this way.
Thus we have the following:
\begin{Thm}\label{thm:building}
	For a modular valuated matroid $(E,\omega)$, 
	the tropical linear space ${\cal T}(\omega)/ \RR{\bf 1}$ 
	is a geometric realization of the Euclidean building 
	associated with the uniform modular lattice ${\cal L}(\omega)$.
\end{Thm}
Dress and Terhalle~\cite{DressTerhalle98} 
claimed this result on the Euclidean building for ${\rm SL} (F^n)$, 
where $F$ is a field with a discrete valuation. 
In the previous example, take the whole set $K(t)^n$ as $E$. 
In this case, ${\cal L}(E)$ is 
the lattice of all full-rank free $K(t)^-$-submodules of $K(t)^n$, 
and is a uniform modular lattice of uniform-rank $n$; 
see \cite[Example 3.3]{HH18a}.
In particular, 
valuated matroid $(E, \omega^E)$ is a modular valuated matroid.
The simplicial complex ${\cal C}({\cal L}(E))$ 
is nothing but the Euclidean building for ${\rm SL}(K(t)^n)$; see \cite[Section 19]{Garrett}.

\section*{Acknowledgments}
The author thanks Kazuo Murota, Yuni Iwamasa, and Koyo Hayashi 
for careful reading and helpful comments, and also thanks the referees for helpful comments.
This work was partially supported by JSPS KAKENHI Grant Numbers 
JP25280004, JP26330023, JP26280004, JP17K00029.


\begin{thebibliography}{1}
	\small
	\bibitem{Aigner}
	M. Aigner: {\em Combinatorial Theory}, Springer, Berlin, 1979.
	
     
     \bibitem{ArdilaKlivans06}
	 F. Ardila and C. J. Klivans: 
	 The Bergman complex of a matroid and phylogenetic trees.
	 {\em Journal of Combinatorial Theory. Series B} {\bf 96} (2006), 38--49.
	
	
	\bibitem{BruhatTits}
	F. Bruhat, and J. Tits:
	Groupes r\'eductifs sur un corps local.
	{\em Institut des Hautes \'Etudes Scientifiques. Publications Math\'ematiques} {\bf 41}, (1972), 5--251.
	
	\bibitem{Birkhoff}
	G. Birkhoff: {\it Lattice Theory}. 
	American Mathematical Society, New York, 1940; 3rd edn., 
	American Mathematical Society, Providence, RI, 1967.
	


	
	\bibitem{DressKahrstromMoulton11}
	A. Dress,   J. K\aa hrstr\"om, and V. Moulton:
	A `non-additive' characterization of {$\wp$}-adic norms.
	{\em Annals of Combinatorics} {\bf 15} (2011), 37--50.
	
	
	\bibitem{DressTerhalle93}
	 A. W. M. Dress and W. Terhalle: 
	 A combinatorial approach to $p$-adic geometry. 
	 {\em Geometriae Dedicata} {\bf 46} (1993), 127--148
	 .
	\bibitem{DressTerhalle98}
	 A. Dress and W. Terhalle: 
	 The tree of life and other affine buildings. In:
	 {\em Proceedings of the International Congress of Mathematicians, Vol. III (Berlin, 1998).} 
	 {\em Documenta Mathematica Extra Vol. III}  (1998), 565--574.
	
	\bibitem{DressWenzel_greedy}
	 A. W. M. Dress and W. Wenzel: 
	 Valuated matroids: a new look at the greedy algorithm. 
	 {\em Applied Mathematics Letters} {\bf 3} (1990), 33--35. 
	
	\bibitem{DressWenzel}
	 A. W. M. Dress and W. Wenzel: 
	 Valuated matroids. {\em Advances in Mathematics} {\bf 93} (1992), 
	 214--250.
	
	
	\bibitem{DevelinSturmfels}
	M. Develin and B. Sturmfels: Tropical convexity. 
	{\em Documenta Mathematica} {\bf 9} (2004), 1--27.
	
	
	\bibitem{Garrett}
	P. B. Garrett: {\em Building and Classical Groups}. Chapman \& Hall, London, 	
	1997.
	
	\bibitem{Hampe15}
	 S. Hampe: Tropical linear spaces and tropical convexity.
	 {\em Electronic Journal of Combinatorics} {\bf  22} (2015),  
	 Paper 4.43. 

	\bibitem{HH18a}
	H. Hirai: Uniform modular lattice and Euclidean building. preprint, 2017, {\tt arXiv:1801.00240}. 
	
	\bibitem{HH18c}
	H. Hirai: Computing degree of determinant via discrete convex optimization on 
	Euclidean building. preprint, 2018, {\tt arXiv:1805.11245}.
	
	\bibitem{JoswigSturmfelsYu07}
	M. Joswig,  B. Sturmfels, and J. Yu:
	Affine buildings and tropical convexity. 
	{\em Albanian Journal of Mathematics} {\bf 1} (2007),  187--211. 
	
	\bibitem{KeelTevelev06}
	S. Keel and J. Tevelev: Geometry of Chow quotients of Grassmannians, 
	{\em Duke Mathematical Journal} {\bf 134} (2006), 259--311.
	
	\bibitem{TropicalBook}
	D. Maclagan and B. Sturmfels:
	{\it Introduction to Tropical Geometry}.
	American Mathematical Society, Providence, RI, 2015.
	
	
    \bibitem{Murota97}
	K. Murota: Characterizing a valuated delta-matroid as a family of
	delta-matroids.
	{\em Journal of the Operations Research Society of Japan} {\bf 40}, (1997), 565--578.  
	
	\bibitem{MurotaMatrix}
	K. Murota: {\it Matrices and Matroids for
		Systems Analysis.} Springer-Verlag, Berlin, 2000.
	
	\bibitem{MurotaBook}
	K. Murota: 
	{\it Discrete Convex Analysis.}
	SIAM, Philadelphia, 2003.
	
	\bibitem{MurotaTamura01}
	K. Murota and A. Tamura: 
	On circuit valuation of matroids. 
	{\em Advances in Applied Mathematics} {\bf 26} (2001), 192--225.

	 
	 \bibitem{Speyer08}
	  D. E. Speyer: Tropical linear spaces. 
	  {\em SIAM Journal on Discrete Mathematics} {\bf 22} (2008), 1527--1558.
	 
	 \bibitem{SpeyerSturmfels04}
	 D. Speyer and B. Sturmfels: 
	 The tropical Grassmannian. {\em Advances in Geometry} {\bf  4} (2004), 389--411.

	
	\bibitem{Tits}
	J. Tits: {\it Buildings of Spherical Type and Finite BN-pairs}. 
	Lecture Notes in Mathematics, Vol. 386. Springer-Verlag, Berlin-New York, 1974.
	
	\bibitem{Zhang2018}
	L. Zhang: Computing convex hulls in the affine building of $SL_d$, 
	preprint, 2018,  {\tt arXiv:1811.08884}.
	

\end{thebibliography}
\end{document}